\documentclass[11pt,reqno]{amsart}
\usepackage[brazilian,english]{babel} 
\usepackage[T1]{fontenc}
\usepackage{amsfonts,amsthm,amssymb,amsmath}  
\usepackage{graphicx,color}       
\usepackage{latexsym}       
\usepackage{overpic}        
\usepackage[colorlinks=false, linktocpage=true]{hyperref}
\usepackage{graphicx}
\usepackage{multicol}
\usepackage{color}
\usepackage{multirow} 

\newcommand{\leqnomode}{\tagsleft@true}
\newcommand{\reqnomode}{\tagsleft@false}
\makeatother

\newtheorem{teore}{Theorem}[section]
\newtheorem{obs}[teore]{Remark}
\newtheorem{defi}[teore]{Definition}

\newtheorem{coro}[teore]{Corollary}

\newtheorem{lem}[teore]{Lemma}
\newcommand{\ka}{\kappa}
\newcommand{\R}{\mathbb{R}}

\newcommand{\C}{\mathbb{C}}

\newcommand{\psib}{\mbox{\boldmath$\psi$}}

\newcommand{\ub}{\mathbf{u}}

\newcommand{\zb}{\mathbf{z}}

\usepackage[pagewise]{lineno}

\numberwithin{equation}{section}

\title[NLS systems with $p$-type nonlinearities]{Blow-up solutions to a class of nonlinear coupled Schr\"odinger systems  with  power-type-growth nonlinearities}

\author[N. Noguera]{Norman Noguera}
\address{SM-UCR, Ciudad Universitaria Carlos Monge Alfaro, Departamento de Ciencias Naturales, Apdo: 111-4250, San Ram\'on, Alajuela, Costa Rica}
\email{norman.noguera@ucr.ac.cr}

\begin{document}

\begin{abstract}
In this work we consider a system of nonlinear Schr\"{o}dinger equations whose  nonlinearities satisfy a power-type-growth. First, we prove that the Cauchy problem  is local and global well-posedness in $L^2$ and $H^1$. Next, we establish the existence of ground state solutions. Then we use these solutions to study the dichotomy of global existence versus blow-up in finite time. Similar results were presented in the reference \cite{NoPa2} for the special case when the growth of the nonlinearities was quadratic. Here we will extend them to systems with nolinearities of  order $p$ (cubic, quartic and so on). Finally, we recover some known results for two particular systems, one with quadratic and the other with cubic growth nolinearities. 
\end{abstract}



\maketitle	

\tableofcontents

\section{Introduction}

In recent years, a great deal of work has been devoted to the study of the dynamics of nonlinear coupled Schr\"{o}dinger systems whose nonlinearities satisfy a certain 

\begin{equation}\label{condp}
\mbox{power-like $p$-ic growth.}
\end{equation}

As in the case of the single  Schr\"{o}dinger equation the behavior of the nonlinearities, such  as \eqref{condp}, plays a fundamental roll to  address the study. Several authors have devoted themselves to the study of systems whose  nonlinearities satisfy a  quadratic growth, that is $p=2$ in \eqref{condp}, to cite a few consider the works, \cite{Pastor},\cite{hamano2018global}, \cite{inui2020blow}, \cite{inui2019scattering}, \cite{mengxu2020}, \cite{wang2021} and references therein. A particular model of system  with quadratic nonlinearities, which  has been studied lately  by many authors,  from a mathematical point of view, is the following
 \begin{equation}\label{system1J}
\begin{cases}
\displaystyle i\partial_{t}u+\Delta u=-2\overline{u}v,\\
\displaystyle i\partial_{t}v+\kappa\Delta v=- u^{2},
\end{cases}
\end{equation}
where $\kappa>0$ is a constant.  It appeared for the first time in reference  \cite{Hayashi}. There, the Cauchy problem with initial data in $L^2$, $H^1$ and weighted-$L^2$ spaces was studied in dimensions $1\leq n\leq 6$. For this system, dimensions  $n=4$ and $n=6$ correspond to the critical cases for $L^2$ and $H^1$ spaces, respectively. For \eqref{system1J},  in \cite{Hayashi} the local and global well-posednees theories were established. It was also proved existence of ground state solutions in dimensions $1\leq n\leq 6$. The authors  also gave a sufficient sharp condition for global existence solutions in $n=4$. In the $L^2$ supercritical and $H^1$ subcritical case, $n=5$, in  references \cite{hamano2018global}  and  \cite{NoPa} the dichotomy of blow-up versus existence of global solutions was studied. A lot of works around system \eqref{system1J} can be found in the literature. These works include, among other results, stability  and instability of standing waves solutions, existence of scattering  and blow-up solutions (see  for example \cite{zhang2018stable},  \cite{dinh2020existence}, \cite{dinh2020instability},\cite{inui2019scattering}, \cite{inui2020blow} and \cite{hamano2018global}).\\

A   generalization of system \eqref{system1J} was considered in reference \cite{NoPa2}. The authors studied the following system 

\begin{equation}\label{system10}
\begin{cases}
\displaystyle i\alpha_{k}\partial_{t}u_{k}+\gamma_{k}\Delta u_{k}-\beta_{k} u_{k}=-f_{k}(u_{1},\ldots,u_{l})\\
(u_{1}(x,0),\ldots,u_{l}(x,0))=(u_{10},\ldots,u_{l0}),\qquad k=1,\ldots l,
\end{cases}
\end{equation}
where $u_{k}:\R^{n}\times \R\to \C$, $(x,t)\in \R^{n}\times \R$, $\alpha_{k}, \gamma_{k}>0$, $\beta_{k}\geq0$ are real constant and the nonlinearities, $f_{k}$, satisfy \eqref{condp} with   $p=2$, which  fundamentally is represented as the following property
\begin{equation}\label{growth2}
\left|f_{k}(\mathbf{z})\right|\leq C\sum_{j=1}^{l}|z_{j}|^{2}, \qquad k=1\ldots, l.
\end{equation}
The authors proposed a group of assumptions that allowed them to study certain aspect of the dynamic of  \eqref{system10} and also to achieved \eqref{growth2}. More specifically, was consider system \eqref{system10} endowed with the following assumptions.

\newtheorem{thmx}{}
\renewcommand\thethmx{(H1)}
\begin{thmx}\label{H1}
	We have
	\begin{align*}
		f_{k}(\mathbf{0})=0, \qquad  k=1,\ldots,l.
	\end{align*}
\end{thmx}

\renewcommand\thethmx{(H2)}
\begin{thmx}\label{H2}
 For any $\mathbf{z},\mathbf{z}'\in \C^{l}$  we have
	\begin{equation*}
		\begin{split}
			\left|\frac{\partial }{\partial z_{m}}[f_{k}(\mathbf{z})-f_{k}(\mathbf{z}')]\right|+ \left|\frac{\partial }{\partial \overline{z}_{m}}[f_{k}(\mathbf{z})-f_{k}(\mathbf{z}')]\right|&\lesssim \sum_{j=1}^{l}|z_{j}-z_{j}'|,\qquad k,m=1,\ldots,l,
		\end{split}
	\end{equation*}
\end{thmx}

\renewcommand\thethmx{(H3)}
\begin{thmx}\label{H3}
	There exists a function $F:\C^{l}\to \C$,  such that
	\begin{equation*}
		f_{k}(\mathbf{z})=\frac{\partial F}{\partial \overline{z}_{k}}(\mathbf{z})+\overline{\frac{\partial F }{\partial z_{k}}}(\mathbf{z}),\qquad k=1\ldots,l.
	\end{equation*}
\end{thmx}

\renewcommand\thethmx{(H4)}
\begin{thmx}\label{H4}
 For any $ \theta \in \R$ and $\mathbf{z}\in \mathbb{C}^{l}$,
\begin{equation*}
\mathrm{Re}\,F\left(e^{i\frac{\alpha_{1}}{\gamma_{1}}\theta  }z_{1},\ldots,e^{i\frac{\alpha_{l}}{\gamma_{l}}\theta  }z_{l}\right)=\mathrm{Re}\,F(\mathbf{z}).
	\end{equation*}	
\end{thmx}

\renewcommand\thethmx{(H5)}
\begin{thmx}\label{H5}
	Function $F$ is homogeneous of degree $3$, that is, for any $\mathbf{z}\in \mathbb{C}^{l}$ and $\lambda >0$,
	\begin{equation*}
		F(\lambda \mathbf{z})=\lambda^{3}F(\mathbf{z}).
	\end{equation*}
\end{thmx}

\renewcommand\thethmx{(H6)}
\begin{thmx}\label{H6}
	There holds
	\begin{equation*}
		\left|\mathrm{Re}\int_{\R^{n}} F(\ub)\;dx\right|\leq \int_{\R^{n}} F(\!\!\big\bracevert\!\! \mathbf{u}\!\!\big\bracevert\!\!)\;dx.
	\end{equation*}
\end{thmx}

\renewcommand\thethmx{(H7)}
\begin{thmx}\label{H7}
	Function $F$ is real valued on $\R^l$, that is, if $(y_{1},\ldots,y_{l})\in \R^{l}$ then
	\begin{equation*}
		F(y_{1},\ldots,y_{l})\in \R.
	\end{equation*}
	Moreover, functions	$f_k$ are non-negative on the positive cone in $\mathbb{R}^l$, that is, for $y_i\geq0$, $i=1,\ldots,l$,
	\begin{equation*}
		f_{k}(y_{1},\ldots,y_{l})\geq0.
	\end{equation*}	
\end{thmx}

\renewcommand\thethmx{(H8)}
\begin{thmx}\label{H8}
	Function $F$ can be written as the sum $F=F_1+\cdots+F_m$, where $F_s$, $s=1,\ldots, m$ is super-modular on $\R^d_+$, $1\leq d\leq l$ and vanishes on hyperplanes, that is, for any $i,j\in\{1,\ldots,d\}$, $i\neq j$ and $k,h>0$, we have
	\begin{equation*}
		F_s(y+he_i+ke_j)+F_s(y)\geq F_s(y+he_i)+F_s(y+ke_j), \qquad y\in \R^d_+,
	\end{equation*}
	and $F_s(y_1,\ldots,y_d)=0$ if $y_j=0$ for some $j\in\{1,\ldots,d\}$.
\end{thmx}

Under assumptions \ref{H1}-\ref{H8}, in reference \cite{NoPa2}, was proved local and global existence in  $L^{2}$ y $H^{1}$ spaces,  in dimensions $1\leq n\leq 5$ . Also was proved existence of standing waves solutions analyzing the ground states solutions and conditions to get global and blow-up solutions in terms of the ground state were established. The scattering problem, in dimension $n=5$, of this system in the case of mass-resonance (see Section \ref{sec.masreso}) was considered in \cite{NoPa4}. The mass-resonance assumption in the scattering problem was dropped in reference  \cite{NoPa5}. Finally, the existence of ground state solutions in the critical case, that is, $n=6$ was studied in \cite{NoPa3}.  \\

 Systems such as \eqref{system1J} or \eqref{system10}, with   $p=2$ in \eqref{condp}, appear in  physics phenomena in nonlinear optics, for example  the Second Harmonic Generation (SHG) or frequency doubling. These models arise when a light beam interact with a material who has a  $\chi^{(2)}$ response. Basically, two photons interact to form new photons having twice the frequency of those initial photons. As an application of this nonlinear process it is possible to obtain imaging collagen fibers in organs such as lung, kidney, and liver as well as in connecting
tissues such as tendon, skin, bones, and blood vessels (see \cite[Chapter 5]{Kumar2018}).\\

Similarly, in the literature we find coupled nonlinear Schr\"{o}dinger systems with nonlinearities satisfying a cubic growth, that is  $p=3$ in \eqref{condp}, arising in physical phenomena.  These systems appear, for example,  by describing a light propagation in material with Kerr ($\chi^{(3)}$) response. A particular model of such a systems is given by 
\begin{equation}\label{systemFA}
\begin{cases}
\displaystyle i\partial_{t}u+\Delta u-u+\left(\frac{1}{9}|u|^{2}+2|w|^{2}\right)u+\frac{1}{3}\overline{u}^{2}w=0,\\
\displaystyle i\sigma\partial_{t}w+\Delta w-\mu w+ (9|w|^{2}+2|u|^{2})w+\frac{1}{9}u^{3}=0,
\end{cases}
\end{equation}
where $\sigma,\mu>0$. In this case, we have $p=3$ in \eqref{condp}. The one dimensional version of the system \eqref{systemFA} was derived in \cite{sammut1998bright}. The local and global Cauchy problem for periodic initial data was studied in reference \cite{angulo2009stability}. There were also obtained results related with  linear and nonlinear stability. In reference \cite{oliveira2018schr} system \eqref{systemFA}  was considered  with initial data in $H^1$ for dimensions  $1\leq n\leq 3$. The authors studied the local and global problem,  proved existence of ground state solutions, analyzed their stability and gave some criteria for the existence of blow-up solutions.   More recently, in reference \cite{ardila2021sharp}   was proved a scattering result in $H^1(\R^3)$ for the radial and nonradial case.   \\

\newpage

  The aim of this work is to study the dichotomy of global existence versus blow-up in finite time of the system

  \begin{equation}\label{system1}
\begin{cases}
\displaystyle i\alpha_{k}\partial_{t}u_{k}+\gamma_{k}\Delta u_{k}-\beta_{k} u_{k}=-f_{k}(u_{1},\ldots,u_{l})\\
(u_{1}(x,0),\ldots,u_{l}(x,0))=(u_{10},\ldots,u_{l0}),\qquad k=1,\ldots l,
\end{cases}
\end{equation}
by modifying the assumptions,  established in \cite{NoPa2} to obtain a more general system.  Since here we consider a power-type-growth of order $p$, we simultaneously study a broad class of nonlinearities ( and hence nonlinear Schr\"{o}dinger systems) such as \eqref{system1J} and \eqref{systemFA}. Specifically, compared to this previous work,  instead of assumptions \ref{H2}, \ref{H4} and \ref{H5} here we assume,
  \renewcommand\thethmx{(H2*)}
\begin{thmx}\label{H2*}
Let $p>1$. For any $\mathbf{z},\mathbf{z}'\in \C^{l}$  we have
	\begin{equation*}
		\begin{split}
			\left|\frac{\partial }{\partial z_{m}}[f_{k}(\mathbf{z})-f_{k}(\mathbf{z}')]\right|+ \left|\frac{\partial }{\partial \overline{z}_{m}}[f_{k}(\mathbf{z})-f_{k}(\mathbf{z}')]\right|&\lesssim \sum_{j=1}^{l}|z_{j}-z_{j}'|^{p-1},\qquad k,m=1,\ldots,l,
		\end{split}
	\end{equation*}
\end{thmx}
  
  \renewcommand\thethmx{(H4*)}
\begin{thmx}\label{H4*}
There exist positive constants $\sigma_{1},\ldots,\sigma_{l}$ such that for any $\mathbf{z}\in \mathbb{C}^{l}$
\begin{equation*}
\mathrm{Im}\sum_{k=1}^{l}\sigma_{k}f_{k}(\mathbf{z})\overline{z}_{k}=0 .
	\end{equation*}	
\end{thmx}

\renewcommand\thethmx{(H5*)}
\begin{thmx}\label{H5*}
	Function $F$ is homogeneous of degree $p+1$, that is, for any $\mathbf{z}\in \mathbb{C}^{l}$ and $\lambda >0$,
	\begin{equation*}
		F(\lambda \mathbf{z})=\lambda^{p+1}F(\mathbf{z}).
	\end{equation*}
\end{thmx}

\begin{obs}
Here some comments about these hypotheses. 
\begin{enumerate}
    \item[(i)] Assumptions \textnormal{\ref{H2*}} and \textnormal{\ref{H5*}},  are related to the behavior of nonlinearities $f_k$. In fact with these hypothesis we can guarantee
    \eqref{condp}, that is, $f_k$ satisfies (see Lemma \ref{estdifF}) for $p>1$
    \begin{equation*}
\left|f_{k}(\mathbf{z})\right|\leq C\sum_{j=1}^{l}|z_{j}|^{p}, \qquad k=1\ldots, l.
\end{equation*}
    \item[(ii)]  Assumption \textnormal{\ref{H4*}}  assure that \textit{a priori},  system \eqref{system1} does not satisfy the mass resonance condition (see condition \eqref{RC} also reference \cite{NoPa3}).
\end{enumerate}
    
\end{obs}
  
 Under assumptions \ref{H1}-\ref{H8}, with   \ref{H2*}, \ref{H4*} and \ref{H5*}  instead of  \ref{H2}, \ref{H4} and \ref{H5}, we give a sufficient condition for global well-posedness in $H^1$ of \eqref{system1} with and without mass-resonance,  and considering the mass resonance framework we show that such condition is sharp (see Theorems \ref{thm:globalexistencecondn=51} and \ref{thm:globalexistencecondn=52}). The conditions given in Theorems \ref{thm:globalexistencecondn=51} and \ref{thm:globalexistencecondn=52} depend not only of the dimension but also of the value of $p$. Particularly, Theorem \ref{thm:globalexistencecondn=52} is written in terms of the critical Sobolev index $s_{c}=\frac{n}{2}-\frac{2}{p-1}$.\\

An explanation of the importance of hypotheses  \ref{H1}-\ref{H8} and, hence \ref{H2*} and \ref{H4*}-\ref{H5*}, in the study of the dynamics of 
 a system such as \eqref{system1} can be found in the reference \cite{NoPa2}. Here we only mention that if $F_s$ is $C^2$ then \ref{H8} is equivalent to (see \cite{haj})
\begin{equation}\label{condsupmod}
\frac{\partial^2 F}{\partial{x_j}\partial{x_i}}\geq0.
\end{equation}


It is important to note that most of the nonlinearities in models with nonlinear response, for example   $\chi^{(2)}$ or $\chi^{(3)}$ have a polynomial form such as systems \eqref{system1J} and \eqref{system1}. However, there are systems which can show  non-polynomial nonlinearities, for instance, rational functions, \cite{Yew}. So these cases can be included in the study of a system like \eqref{system1}. On the other hand, as we can see,  the nonlinearities $f_k$ have no explicit form. This  allows us to study simultaneously and in a unified way a larger number of systems. \\

As we will see (section \ref{sec.gsbu}), the flow of  system \eqref{system1} conserves the mass and the energy  given by
\begin{equation}\label{conservationcharge}	Q(\mathbf{u}(t)):=\sum_{k=1}^{l}\frac{\sigma_{k}\alpha_{k}}{2}\|  u_{k}(t)\|_{L^{2}}^{2},
	\end{equation}
	
	and 
	\begin{equation}
\label{conservationenergy}
	E(\mathbf{u}(t)):=K(\mathbf{u})+L(\mathbf{u})
    -2P(\mathbf{u}),
\end{equation}
where $K$, $L$ and $P$ are defined by
\begin{equation}\label{funcK1}
K(\mathbf{u})=\sum_{k=1}^{l}\gamma_{k}\|\nabla u_{k}\|_{L^2}^{2}, \quad L(\mathbf{u})=\sum_{k=1}^{l}\beta_{k}\|u_{k}\|_{L^2}^{2} \quad \mbox{and} \quad P(\mathbf{u})=\mbox{Re}\int F(\mathbf{u})\;dx
\end{equation}

 Note  that the $\sigma_k$ in \eqref{conservationcharge} are the same  appearing in assumption \ref{H4}. \\

We also study the existence of standing waves solutions. These are solutions of system \eqref{system1} of the form 

\begin{equation}\label{standing}
u_{k}(x,t)=e^{i\frac{\sigma_{k}}{2}\omega t}\psi_{k}(x),\qquad k=1,\ldots,l,
\end{equation}
where the functions $\psi_{k}$ are real and decay to zero at infinity. By replacing \eqref{standing} into \eqref{system1}, we will see that  $\psi_{k}$ must satisfy the following nonlinear elliptic system
\begin{equation}\label{system3}
\displaystyle -\gamma_{k}\Delta \psi_{k}+\left(\frac{\sigma_{k}\alpha_{k}}{2}\omega+\beta_{k}\right) \psi_{k}=f_{k}(\psib),\qquad k=1,\ldots,l.
\end{equation}

The action functional associated to system \eqref{system3} is defined by
\begin{equation}\label{FunctionalI1}
I(\psib)=\frac{1}{2}\left[K(\psib)+\mathcal{Q}(\psib)\right]-P(\psib),
\end{equation}
where
\begin{equation}
\label{functionalQ}
\mathcal{Q}(\psib):=\sum_{k=1}^{l}\left(\frac{\sigma_{k}\alpha_{k}}{2}\omega+\beta_{k}\right)\| \psi_{k}\|_{L^2}^{2}.
\end{equation}

In order to find standing wave solutions for  system \eqref{system1}  we have to look for a very special type of solutions of system \eqref{system3}, actually weak solutions, called \textit{ground state}. These are nontrivial critical points of functional  $I$ which also minimizes it (see section \ref{sec.gs}). We denote  by $\mathcal{G}(\omega,\boldsymbol{\beta})$ the set of all ground states for system \eqref{system3}, where  $(\omega,\boldsymbol{\beta})$ indicates the dependence on the parameters $\omega$ and $\boldsymbol{\beta}=(\beta_1,\ldots,\beta_l)$.
\\

\subsection{Mass resonance condition}\label{sec.masreso}
A property that influences the long time behavior of the solutions of systems like \eqref{system1} is the mass-resonance condition. For system \eqref{system1J} this condition is reflected in the value of the parameter $\kappa$. In fact, system \eqref{system1J} satisfies the mass-resonance condition if $\displaystyle \kappa =\frac{1}{2}$ (see   \cite{Hayashi}). On the other hand, we say that system \eqref{systemFA} satisfies the mass-resonance condition if $\displaystyle \sigma =3$ (see   \cite{oliveira2018schr}). For a system like \eqref{system1}, we say that it  is in mass-resonance condition if
\begin{equation}\label{RC}
\mathrm{Im}\sum_{k=1}^{l}\frac{\alpha_{k}}{\gamma_{k}}f_{k}(\zb)\overline{z}_{k}=0, \quad \zb\in \mathbb{C}^{l},
\end{equation}
where $\alpha_k$ and $\gamma_k$ are the parameters appearing in the system (see \cite[Definition 1.1]{NoPa3}).\\

\begin{obs}
    Note the similarity between \eqref{RC} and assumption \textnormal{\ref{H4}}. In principle, we are not assuming any relation between the value $\sigma_k$ and the quotient $\displaystyle \frac{\alpha_k}{\gamma_k}$, $k=1,\ldots, l$. But, for example when there exists a positive constant $\rho$ such that $\sigma_k=\rho\displaystyle \frac{\alpha_k}{\gamma_k}$ for each $k=1,\ldots, l$, then  system \eqref{system1} is, in fact, in mass resonance condition. 
\end{obs}
As we mentioned, to be  in mass-resonance condition or not affects directly how to study the long time behavior of the solutions. For example,  in order to prove the existence of scattering and blow-solutions, an extra effort is needed by trying to control some terms appearing in the virial identities. In fact, if we define the variance 
\begin{equation}\label{variancemk}
V(t):=\int |x|^{2}\left(\sum_{k=1}^{l}\frac{\alpha_{k}^{2}}{\gamma_{k}}|u_{k}|^{2}\right)\;dx.
\end{equation}
we obtain (see section \ref{sec.gsbu})
\begin{equation}\label{V1mwmr}
\begin {split}
    V'(t)&=4\sum_{k=1}^{l}\alpha_{k}\mathrm{Im}\int x\cdot \nabla u_{k} \overline{u}_{k}\;dx-4\int |x|^{2}\mathrm{Im}\sum_{k=1}^{l}\frac{\alpha_{k}}{\gamma_{k}}f_{k}(\mathbf{u})\overline{u}_{k}\;dx
\end{split}
\end{equation}
and
\begin{equation}\label{V2mwmr}
\begin{split}
    V''(t)&=2 n(p-1)E(\mathbf{u}_0)-2 n (p-1) L(\mathbf{u})+2 (4-np+n)K(\mathbf{u})\\
    &\quad-4\frac{d}{dt}\left[\int |x|^{2}\mathrm{Im}\sum_{k=1}^{l}\frac{\alpha_{k}}{\gamma_{k}}f_{k}(\mathbf{u})\overline{u}_{k}\;dx\right].
\end{split}
\end{equation}

Thus, if system \eqref{system1} is under mass-resonance the last term appearing in \eqref{V1mwmr} and \eqref{V2mwmr} vanishes and we can apply a standard convex argument to get blow-up solutions. \\

Other properties which commonly are satisfied by the single Nonlinear Schr\"{o}dinger Equation are the   Galilean and the pseudo-conformal invariance (in the $L^2$-critical case). However, in the case of system \eqref{system1} this is true only when it is  under the mass-resonance condition. In section \ref{subconformal} we discuss the case of the pseudo-conformal transformation. \\

In the last years, the interest  in studying the dynamics of systems without mass-resonance condition has increased. Readers interested in  results regardeding  systems like \eqref{system1J} when $\displaystyle\kappa\neq \frac{1}{2}$ can see references \cite{hamano2019scattering}, \cite{inui2020blow}, \cite{inui2019scattering} and \cite{wang2019Sacttering}.    For   a more general case of systems with quadratic type nonlinearities which not satisfy the mass-resonance condition, see references \cite{NoPa3} and \cite{NoPa5}. In the case of system \eqref{systemFA}, a result in the non-mass-resonance case can be found in reference \cite{ardila2021sharp}. \\

\subsection{Main results}

Our main results settle the dichotomy between global well-posedness and existence of blow-up solutions for system \eqref{system1} under assumptions \ref{H1}-\ref{H8}. 
 \begin{teore}\label{thm:globalexistencecondn=51}
 	Assume $\mathbf{u}_0\in \mathbf{H}^{1}$ and let  $\mathbf{u}$ be the  solution of \eqref{system1} defined in the maximal  existence  interval $I$. 
 	\begin{enumerate}
  \item[(i)] If $1 < p<1+\frac{4}{n} $, then for any $\mathbf{u}_0\in \mathbf{H}^{1}$,  system \eqref{system1} has a unique  solution which can be extended globally in time, that is  $I=(-\infty,\infty)$.
 		\item[(ii)] Let $\psib \in \mathcal{G}(1,\boldsymbol{0}) $ and $p=1+\frac{4}{n}$. 

   \begin{enumerate}
       \item[a)]    If
 		\begin{equation}\label{L2GSCond}
 		Q(\mathbf{u}_0)<Q(\psib),	
 		\end{equation}
 		then the conclusion in (i) also holds.	
    \item[b]  Moreover,  under the mass-resonance condition, \eqref{RC}  exists an initial data $\mathbf{u}_0$ such that $Q(\mathbf{u}_0)=Q(\psib)$ and the corresponding solution blows up in finite time. 
   \end{enumerate}

 	\end{enumerate}
 
 \end{teore}

In the $H^1$-intercritical case, we also have a sufficient condition for global existence solutions. We describe such a condition in the following theorem.
 \begin{teore}\label{thm:globalexistencecondn=52}
 	Assume $\mathbf{u}_0\in \mathbf{H}^{1}$ and let  $\mathbf{u}$ be the  solution of \eqref{system1} defined in the maximal  existence  interval $I$.
 	 Let $s_{c}=\frac{n}{2}-\frac{2}{p-1}$ and $\psib \in \mathcal{G}(1,\boldsymbol{0}) $. Assume  $1+\frac{4}{n}<p<\frac{n+2}{n-2}$ and in addition that  
 		\begin{equation}\label{conditionsharp1}
 		Q(\mathbf{u}_0)^{1-s_{c}}E(\mathbf{u}_0)^{s_{c}}<Q(\psib)^{1-s_{c}}\mathcal{E}(\psib)^{s_{c}},
 		\end{equation}
 		where $\mathcal{E}$ is the energy defined in \eqref{conservationenergy} with $\boldsymbol{\beta}=\boldsymbol{0}$.\\
 		\begin{enumerate}
 		    \item[(i)]  If 
 		\begin{equation}\label{conditionsharp2}
 		Q(\mathbf{u}_0)^{1-s_{c}}K(\mathbf{u}_0)^{s_{c}}<Q(\psib)^{1-s_{c}}K(\psib)^{s_{c}},
 		\end{equation}
 		then
 		\begin{equation*}
 		Q(\mathbf{u}_0)^{1-s_{c}}K(\mathbf{u}(t))^{s_{c}}<Q(\psib)^{1-s_{c}}K(\psib)^{s_{c}},\qquad t\in I.
 		\end{equation*}
 		In particular the initial value problem \eqref{system1} is globally well-posed in $\mathbf{H}^{1}$.

       \item[(ii)] Under the mass-resonance condition, \eqref{RC}, if  \begin{equation}\label{gradientcondblowup}
Q(\mathbf{u}_0)^{1-s_c}K(\mathbf{u}_0)^{s_c}>Q(\psib)^{1-s_c}K(\psib)^{s_c}.
\end{equation}
and
\begin{itemize}
    \item[a)] $x\mathbf{u}_0\in \mathbf{L}^{2}$ or;
    \item[b)]  $\mathbf{u}_0$ is radially symmetric, $1+\frac{4}{n}<p<\min\left\{\frac{n+2}{n-2},5\right\}$ and $n\geq 2$,
\end{itemize}
  then $I$ is finite. In particular, the solution blows-up in finite time. 
 		\end{enumerate}
	
 \end{teore}

Here some comments about Theorems \ref{thm:globalexistencecondn=51} and \ref{thm:globalexistencecondn=52}.
 \begin{obs}
 \begin{enumerate}
 \item[(i)] Theorem \ref{thm:globalexistencecondn=51} said that in the $L^2$-subcritical case we can extend  local solutions to the global ones (in time) without any assumption on the initial data. In the $L^2$-critical case, in order to extend the solutions globally in time, we have to make an additional assumption; the mass of the initial data has to be less than the mass of any ground state solution. Finally, always in the $L^2$-critical case,  if the mass of the initial data coincides with the mass of any ground state and we assume the mass resonance condition, then we can construct a solution which blows up in finite time. 
 \item[(ii)] Theorem \ref{thm:globalexistencecondn=52} says that in the $L^2$-supercritical and $H^1$-subcritical case, the extension of local solutions to global can be reached if we consider now, a condition involving the energy and the $L^2$-norm of the gradient of the ground state solutions.  
     \item[(iii)]  We point out the importance of the mass-resonance condition in Theorems \ref{thm:globalexistencecondn=51} (iii) and \ref{thm:globalexistencecondn=52} (ii), since under this condition we can   show the existence of blow up solutions; in the first case, using the pseudo-conformal invariance  and, in the second case, a standard virial type identity; both properties available only   in this scenario. 
 \end{enumerate}
    
 \end{obs}

This paper is organized as follows. In section \ref{sec.conseq} we establish the notation used in the work. Also, we present some  consequences of   \ref{H1}-\ref{H8} and state some technical lemmas needed.  In section \ref{sec.lgt}, we use the results of section \ref{sec.conseq} to  study the local well-posedness  of system \eqref{system1} in the Sobolev spaces $L^{2}$ and  $H^{1}$. This is done using the contraction mapping fixed point theorem by looking for solutions of the system in a suitable space, depending on $p$.  Section \ref{sec.gs} is devoted to establishing  existence of ground state solutions for the associated elliptic system and, as a consequence, existence of standing wave solutions for the system \eqref{system1}.  In section \ref{sec.gsbu} we prove the main results of this work: Theorems \ref{thm:globalexistencecondn=51} and \ref{thm:globalexistencecondn=52}. First we establish some conservation laws satisfied by the solutions of the system \eqref{system1}. Next, under the mass resonance condition we state a ``standard''  and local virial identity and prove the invariance of solutions of the system \eqref{system1} by the pseudo-conformal transformation. To prove part i) and ii.a) of Theorem  \ref{thm:globalexistencecondn=51}, we get an \textit{a priori} estimate to extend the local $H^1$ solutions  globally in time. For part ii.b) we construct  a blow-up solution using the the pseudo-conformal transformation invariance. To prove  part i)  in Theorem  \ref{thm:globalexistencecondn=52} we again obtain an \textit{a priori} estimate of the $L^2$ norm of the gradient of the solution. For part ii)  we use the virial indentities to apply a convex argument in  cases when solutions have a finite variance or are radial. Finally, in section \ref{sec.parti.cases} we recover some existing results in the literature related to the particular systems \eqref{system1J} and \eqref{systemFA}.

\section{Preliminaries}\label{sec.conseq}

In this section we introduce some notations and give some consequences of our assumptions.

\subsection{Notation}
We use $C$ to denote several constants that may vary line-by-line.
Given any set $A$, by  $\mathbf{A}$ we denote  the product  $\displaystyle A\times \cdots \times A $ ($l$ times). In particular, if $A$ is a Banach space with norm $\|\cdot\|$ then $\mathbf{A}$ is also a Banach space with the standard norm given by the sum. Given any complex number $z\in\mathbb{C}$, Re$z$ and Im$z$ represents its real and imaginary parts. Also, $\overline{z}$ denotes its complex conjugate. In $\mathbb{C}^l$ we frequently  write $\mathbf{z}$ and $\mathbf{z}'$ instead of  $(z_{1},\ldots,z_{l})$ and  $(z_{1}',\ldots,z_{l}')$. Given $\mathbf{z}=(z_{1},\ldots,z_{l})\in \mathbb{C}^l$, we write $z_m=x_m+iy_m$ where $x_m$ and $y_m$ are, respectively, the real and imaginary parts of $z_m$. As usual,  the operators $\partial/\partial z_m$ and $\partial/\partial \overline{z}_m$ are defined by
$$
\dfrac{\partial}{\partial z_m}=\frac{1}{2}\left(\frac{\partial}{\partial x_m} -i\frac{\partial}{\partial y_m}\right), \qquad\dfrac{\partial}{\partial \overline{z}_m}=\frac{1}{2}\left(\frac{\partial}{\partial x_m} +i\frac{\partial}{\partial y_m}\right).
$$ 
The spaces
 $L^{p}=L^{p}(\R^{n})$, $1\leq p\leq \infty$, and $W^s_p=W^s_p(\R^{n})$ denotes the usual Lebesgue and Sobolev spaces. In the case $p=2$, we use the standard notation $H^s=W^s_2$. The space  $H^{1}_{rd}=H^{1}_{rd}(\R^{n})$ is the subspace of radially symmetric non-increasing functions in $H^{1}$. 

To simplify notation, if no confusion is caused we use $\int f\, dx$ to denote  $\int_{\R^n} f\, dx$.
Given a time interval $I$, the mixed   spaces $L^p(I;L^q(\R^n))$ are endowed with the norm
$$
\|f\|_{L^p(I;L^q)}=\left(\int_I \left(\int_{\R^n}|f(x,t)|^qdx \right)^{\frac{p}{q}} dt \right)^{\frac{1}{p}},
$$
with the obvious modification if either $p=\infty$ or $q=\infty$. When the interval $I$ is implicit and no confusion will be caused we denote  $L^p(I;L^q(\R^n))$ simply by  $L^p(L^q)$ and its norm by $\|\cdot\|_{L^p(L^q)}$. More generally, if $X$ is a Banach space, $L^{p}(I;X)$ represents the $L^p$ space of $X$-valued functions defined on $I$.

\subsection{Consequences of H's assumptions}
 The following  properties characterize in certain way the nonlinearities, $f_k$ and the potential function appearing in \ref{H5*}. We can proof these properties using similar ideas to those presented in reference \cite{NoPa2}, but taking into account the modifications in assumptions \ref{H2*} and \ref{H5*}.

\begin{lem}\label{limfk}
Assume that \textnormal{\ref{H1}} and \textnormal{\ref{H2*}}  hold, 

\begin{itemize}
    \item[(i)]  Then
    \begin{equation}\label{estdiffkeq}
\begin{split}
\left|f_{k}(\mathbf{z})-f_{k}(\mathbf{z}')\right|&\leq  C \sum_{m=1}^{l}\sum_{j=1}^{l}(|z_{j}|^{p-1}+|z_{j}'|^{p-1})|z_{m}-z_{m}'|, \qquad k=1\ldots, l.
\end{split}
\end{equation}
and
\begin{equation}\label{growthp}
\left|f_{k}(\mathbf{z})\right|\leq C\sum_{j=1}^{l}|z_{j}|^{p}, \qquad k=1\ldots, l.
\end{equation}

\item[(ii)]  Let $\mathbf{u}$ and $\mathbf{u}'$ be complex-valued functions defined on $\R^n$. Then,
\begin{equation*}
 \begin{split}
 |\nabla[f_{k}(\mathbf{u})-f_{k}(\mathbf{u}'
 )]|&\leq  C\sum_{m=1}^{l}\sum_{j=1}^{l}|u_{j}|^{p-1}|\nabla(  u_{m}-u_{m}')|\\
 &\quad+C\sum_{m=1}^{l}\sum_{j=1}^{l}|u_{j}-u_{j}'|^{p-1}|\nabla u_{m}'|.
 \end{split}
 \end{equation*}
 \item[(iii)] Let $1\leq p_{1},q_{1},r_{1}\leq \infty$ and $p>1$ such that $\frac{1}{r_1}=\frac{p-1}{p_1}+\frac{1}{q_1}$. Assume that $\mathbf{u},\mathbf{u'}\in \mathbf{L}^{p_1}(\R^{n}) $ and $\nabla \mathbf{u}, \nabla \mathbf{u'}\in \mathbf{L}^{q_1}(\R^{n}) $. Then, for $k=1,\ldots,l$,
 \begin{equation*}
 \|\nabla[ f_{k}(\mathbf{u})-f_{k}(\mathbf{u'})]\|_{\mathbf{L}^{r_1}}\leq C\|\mathbf{u}\|^{p-1}_{\mathbf{L}^{p_1}}\|\nabla( \mathbf{u}-\mathbf{u'})\|_{\mathbf{L}^{q_1}}+C\|\mathbf{u}-\mathbf{u'}\|^{p-1}_{\mathbf{L}^{p_1}}\|\nabla \mathbf{u'}\|_{\mathbf{L}^{q_1}}.
 \end{equation*}
 In particular, taking $\mathbf{u'}=0$ we have
  \begin{equation*}
 \|\nabla f_{k}(\mathbf{u})\|_{\mathbf{L}^{r_1}}\leq C\|\mathbf{u}\|^{p-1}_{\mathbf{L}^{p_1}}\|\nabla \mathbf{u}\|_{\mathbf{L}^{q_1}}.
 \end{equation*}
\end{itemize}

\end{lem}
 \begin{proof}
 	Part i) is a consequence of  the Fundamental Theorem of  Calculus and assumptions \textnormal{\ref{H1}}- \textnormal{\ref{H2*}}. Part ii) is proved by using the chain rule and  assumptions \textnormal{\ref{H1}}- \textnormal{\ref{H2*}}. Finally, combining part ii) and  H\"older's inequality, we obtain part iii). 
 \end{proof}

 \begin{obs}
 Note that \eqref{growthp} in fact,  gives us the power-like growth required for the nonlinearities $f_k$.
 \end{obs}

 Now we establish some consequences of additionally assuming   \textnormal{\ref{H3}}- \textnormal{\ref{H5*}}.  We start with the following definition. 

\begin{defi}\label{defGC}
We say that functions $f_{k}$ satisfy the \textbf{Gauge condition} if for any $\theta \in \R$,
 \begin{equation}\label{gaugeCon}
f_{k}\left(e^{i\frac{\sigma_{1}}{2}\theta }z_{1},\ldots,e^{i\frac{\sigma_{l}}{2}\theta }z_{l}\right)=e^{i\frac{\sigma_{k}}{2}\theta }f_{k}(z_{1},\ldots,z_{l}),\qquad k=1,\ldots, l.\tag{GC}
\end{equation}
\end{defi}

\begin{lem}\label{H34impGC}
Assume that \textnormal{\ref{H3}} and \textnormal{\ref{H4*}} hold. 

\begin{itemize}
    \item[(i)] Let $\theta \in \R$ and $\zb \in \C^{l}$. Then, 
     \begin{equation*}
\mathrm{Re}\,F\left(e^{i\frac{\sigma_{1}}{2}\theta  }z_{1},\ldots,e^{i\frac{\sigma_{l}}{2}\theta  }z_{l}\right)=\mathrm{Re}\,F(\zb).
\end{equation*}

    \item[(ii)] The nonlinearities $f_k, k=1,\ldots, l$, satisfy the Gauge condition \eqref{gaugeCon}.
    
\end{itemize}

\end{lem}

\begin{proof}
    Part i) is a consequence of the fact that we can write 
    \begin{equation}\label{fksegexp}
        f_{k}(\mathbf{z})
=2\frac{\partial  }{\partial \overline{z}_{k}}\mathrm{Re} \,F(\mathbf{z})
    \end{equation}
 and the chain rule. For part ii) we use part i) and again the chain rule.
\end{proof}

The following result characterizes the potential function $F$ and establishes some relations between the nonlinearities $f_k$ and $F$.

\begin{lem}\label{estdifF} Assume that  \textnormal{\ref{H1}, \ref{H2*}, \ref{H3}} and \textnormal{\ref{H5*}} hold and let $\mathbf{u}$ be a complex-valued function defined on $\R^n$. Then,

\begin{itemize}
    \item[(i)] 
 \begin{equation}\label{estdifFeq}
\begin{split}
\left|\mathrm{Re}\,F(\mathbf{z})-\mathrm{Re}\,F(\mathbf{z}')\right|&\leq  C \sum_{m=1}^{l}\sum_{j=1}^{l}(|z_{j}|^{p}+|z_{j}'|^{p})|z_{m}-z_{m}'|.
\end{split}
\end{equation}
In particular
\begin{equation}\label{estdifFeq1}
|\mathrm{Re}\,F(\mathbf{z})|\leq C \sum_{j=1}^{l}|z_{j}|^{p+1}.
\end{equation}

\item[(ii)]\begin{equation*}
\mathrm{Re}\sum_{k=1}^{l}f_{k}(\mathbf{u})\nabla \overline{u}_{k}=\mathrm{Re}[\nabla F(\mathbf{u})].
\end{equation*}

\item[(iii)] \begin{equation*}
\mathrm{Re}\sum_{k=1}^{l}f_{k}(\mathbf{u})\overline{u}_{k}=\mathrm{Re}[(p+1)F(\mathbf{u})].
\end{equation*}

\item[(iv)]  The nonlinearities $f_{k}$, $k=1,\ldots,l$ are homogeneous  functions of degree $p$.
\end{itemize}
 \end{lem}

 \begin{proof}
To prove part i) we use \eqref{fksegexp} combined with \eqref{estdiffkeq} and the Fundamental Theorem of
Calculus. The inequality \eqref{estdifFeq1} is a particular case of \eqref{estdifFeq}. To prove that in fact $\mathrm{Re}\;F$ vanishes in $\mathbf{0}$, we precisely use \ref{H5*}. 
 \end{proof}

\begin{obs}
\begin{enumerate}
     
 \item [i)]  Lemma \ref{estdifF} i) gives a bound of order $p+1$ to the function  $\mathrm{Re}\,F$ . This is congruent with the fact that $f_k$ in some way is the derivative of the potential function $F$ and $f_{k}$ has growth of order $p$ as we saw in \eqref{growthp}. 
\item[ii)] Properties ii) and iii) in Lemma \ref{estdifF} are important to construct  Virial-type identities.  
\item[iii)] As a consequence of Lemma \ref{estdifF} iv), the system \eqref{system1} (with $\beta_{k}=0$) is invariant by   the scaling
 \begin{equation*}
     u_{k}^{\lambda}(x,t)=\lambda^{
     \frac{2}{p-1}}u_{k}( \lambda x,\lambda^{2} t), \qquad k=1,\ldots,l.
 \end{equation*}
 This means that the Sobolev space $\dot{H}^{s_c}$ is critical, where
 
 \begin{equation}\label{critexp}
    s_c= \frac{n}{2}-\frac{2}{p-1}.
 \end{equation}
  
 Now, consider the following condition for the power $p$. 
  \begin{equation}\label{pcondt}
    \begin{cases}
1<p< \infty,& \quad if \quad  n=1,2;\\
1<p<\frac{n+2}{n-2},& \quad if \quad  n\geq 3.\\
\end{cases} 
\end{equation}
  Then, we have  we  the following regimes: we will say that system \eqref{system1} is
 \begin{equation*}
     L^{2}-
     \begin{cases}
     \mbox{subcritical}, \quad\mbox{if}\quad 1< p< 1+\frac{4}{n},\\
     \mbox{critical},\quad\mbox{if}\quad p=1+\frac{4}{n},\\
     \mbox{supercritical},\quad\mbox{if}\quad p>1+\frac{4}{n};
     \end{cases}\quad \mbox{and}
 \end{equation*}
 \begin{equation*}
       H^{1}-
     \begin{cases}
     \mbox{subcritical},\quad \mbox{if $p$ satisfies \eqref{pcondt}},   \\
     \mbox{critical},\quad\mbox{if}\quad p=\frac{n+2}{n-2},\quad  n\geq 3.\\
     \mbox{supercritical},\quad\mbox{if}\quad p> \frac{n+2}{n-2},\quad  n\geq 3.
     \end{cases}
 \end{equation*}
\end{enumerate}

\end{obs}

The following result said that if we restrict  the values to the real ones, then the relation between the nonlinearities $f_k$ and the potential function $F$ given by \ref{H3} (see also \eqref{fksegexp}). 

\begin{lem}\label{fkreal}
	If  $F$ satisfies \textnormal{\ref{H7}} then 
	\begin{equation*}\label{Cond4.2}
	f_{k}(\mathbf{x})=\frac{\partial F}{\partial x_{k}}(\mathbf{x}), \qquad \forall \mathbf{x}\in \R^{l}.
	\end{equation*}
	In addition,   $F$ is positive on the positive cone of $\R^l$.
\end{lem}
\begin{proof}
	The first part is clear from  \eqref{fksegexp}. For the second part, we use Lemma \ref{estdifFeq1} (iii).
\end{proof}

\subsection{Technical lemmas}
Now we establish some technical lemmas needed to extend the local solutions to global ones and also prove blow-up solutions in the case of radial data. 

\begin{lem}\label{StraussLemaconsequence}
If $f\in H^{1}(\R^{n})$ is a radially symmetric function, then
\begin{equation*}
\|f\|^{p+1}_{L^{p+1}(|x|\geq R)}\leq \frac{C}{R^{(n-1)(p-1)/2}}\|f\|^{(p+3)/2}_{L^{2}(|x|\geq R)}\|\nabla f\|^{(p-1)/2}_{L^{2}(|x|\geq R)}.
\end{equation*}
\end{lem}
\begin{proof}

The proof is a consequence of  Strauss' radial lemma (see also \cite[page 323]{Ogawa}).
\end{proof}

\begin{lem}\label{supercritcalcase}
Let $I$ an open interval with $0\in I$. Let $a\in \R$, $b>0$ and $q>1$. Define $\gamma=(bq)^{-\frac{1}{q-1}}$ and $f(r)=a-r+br^{q}$, for $r\geq 0$. Let $G(t)$ a non-negative continuous  function such that $f\circ G\geq 0$ on $I$. Assume that $a<\left(1-\frac{1}{q}\right)\gamma$.
\begin{enumerate}
\item[(i)] If $G(0)<\gamma$, then $G(t)<\gamma$, $\forall t\in I$.
\item[(ii)] If $G(0)>\gamma$, then $G(t)>\gamma$, $\forall t\in I$.
\end{enumerate}
\end{lem}

\begin{proof}    The proof can be founded  for instance, in \cite{beg}, \cite{Esfahani} or \cite{Pastor}.
\end{proof}

A  slightly modification of part (ii) in Lemma \ref{supercritcalcase} is the following. 

\begin{lem}\label{corosupercritcalcase}
	Let $I$ an open interval with $0\in I$. Let $a\in \R$, $b>0$ and $q>1$. Define $\gamma=(bq)^{-\frac{1}{q-1}}$ and $f(r)=a-r+br^{q}$, for $r\geq 0$. Let $G(t)$ a non-negative continuous  function such that $f\circ G\geq 0$ on $I$. Assume that $a<(1-\delta_{1})\left(1-\frac{1}{q}\right)\gamma$, for some small $\delta_{1}>0$.
	
	If $G(0)>\gamma$, then there exists $\delta_{2}=\delta_{2}(\delta_{1})>0$ such that $G(t)>(1+\delta_{2})\gamma$, $\forall t\in I$.
\end{lem}
\begin{proof}
See \cite[Corollary 3.2]{Pastor}.
\end{proof}

\begin{lem}\label{lemafunctionchi} Let  $r=|x|$, $x\in \R^{n}$. Take $\chi$ to be a smooth function with
\begin{equation*}
 \chi(r)=\left\{\begin{array}{cc}
r^{2},&0\leq r\leq 1\\
c,& r\geq 3,
\end{array}\right.
\end{equation*}
where $c$ is a positive constant and $\chi''(r)\leq 2$, for any $ r\geq 0$. Let $\chi_{R}(r)=R^{2}\chi\left(r/R\right)$.
 \begin{enumerate}
 \item If $r\leq R$, then $\Delta\chi_{R}(r)=2n$ and $\Delta^{2}\chi_{R}(r)=0$.
 \item If $r\geq R$, then
 \begin{equation*}
 \Delta\chi_{R}(r)\leq C_{1},\;\;\;\; \;\;\;\;\;\;\Delta^{2}\chi_{R}(r)\leq \frac{C_{2}}{R^{2}},
 \end{equation*}
 where $C_{1},C_{2}$ are constant depending on $n$.
 \end{enumerate}
\end{lem}

\begin{proof}
The lemma follows by a straightforward calculation.
\end{proof}

 \section{Local and global well-posedness}\label{sec.lgt}

The local and global well-posedness of  system \eqref{system1} in  $L^{2}$  and   $H^{1}$ is obtained by a contraction argument using Strichartz's estimates. First, we study  (\ref{system1}) in $L^2$ with $1< p\leq 1+ \frac{4}{n}$. To do so, we need to define the spaces

\begin{equation*}
X(I)=\begin{cases}
(\mathcal{C}\cap L^{\infty})(I; L^2)\cap L^{q}(I;L^{p+1}),\qquad 1< p<1+\frac{4}{n},\quad q=\frac{4(p+1)}{n(p-1)};\\
(\mathcal{C}\cap L^{\infty})(I; L^2)\cap L^{p+1}(I;L^{p+1}), \qquad p= 1+\frac{4}{n},\\
(\mathcal{C}\cap L^{\infty})(I; L^2)\cap L^{q}(I;L^{p+1}),\qquad  p>1+\frac{4}{n},\quad q=\frac{4(p+1)}{n(p-1)}
\end{cases}
\end{equation*}
for some time interval $I=[-T,T]$ with $T>0$. The norm in $X(I)$ is defined as
\begin{equation*}
\|f\|_{X(I)}=\begin{cases}
\|f\|_{L^{\infty}( L^2)}+ \|f\|_{L^{q}( L^{p+1})},\qquad 1< p< 1+\frac{4}{n},\quad q=\frac{4(p+1)}{n(p-1)};\\
\|f\|_{L^{\infty}( L^2)}+ \|f\|_{L^{p+1}( L^{p+1})}, \qquad p= 1+ \frac{4}{n};\\
\|f\|_{L^{\infty}( L^2)}+ \|f\|_{L^{q}( L^{p+1})},\qquad p> 1+\frac{4}{n},\quad q=\frac{4(p+1)}{n(p-1)}.
\end{cases}
\end{equation*}

Next, we define
\begin{equation*}
    \theta(n,p):=\frac{4-n(p-1)}{4}, \qquad \mbox{if}\quad 1< p\leq 1+\frac{4}{n}
\end{equation*}
and  the operator
$$\Gamma(\mathbf{u})=(\Phi_{1}(\mathbf{u}),\ldots,\Phi_{l}(\mathbf{u})),$$
where 
\begin{equation*}
\Phi_{k}(\mathbf{u})(t)=\displaystyle U_{k}(t)u_{k0}+i\int_{0}^{t}U_{k}(t-t') \frac{1}{\alpha_{k}}f_{k}(\mathbf{u})\;dt',
\end{equation*}
and $U_{k}(t)$ denotes the Schr\"odinger evolution group defined by $\displaystyle U_{k}(t)=e^{i\frac{t}{\alpha_{k}}(\gamma_{k}\Delta-\beta_{k})}$, $ k=1\ldots, l.$

From Strichartz's estimates, Lemma \ref{limfk} i), H\"{o}lder inequality and the definition of $X(I)$  we have for $\mathbf{u},\mathbf{u}' \in  \mathbf{X}(I)$

\begin{equation*}
\begin{split}
\left\|\int_{0}^{t}U_{k}(t-t')\frac{1}{\alpha_{k}}[ f_{k}(\mathbf{u})-f_{k}(\mathbf{u}')]\;dt'\right\|_{X(I)}&\leq CT^{\theta(n,p)}\left(\|\mathbf{u}\|^{p-1}_{\mathbf{X}(I)}+\|\mathbf{u}'\|^{p-1}_{\mathbf{X}(I)}\right)\|\mathbf{u}-\mathbf{u}'\|_{\mathbf{X}(I)},
\end{split}
\end{equation*}
for some time interval $I$.

Then we can prove existence of local solutions in the case $1<p<1+\frac{4}{n}$ by  the contraction mapping principle applied on the closed ball
	$$B(T,a)=\left\{\mathbf{u}\in \mathbf{X}(I):\|\mathbf{u}\|_{\mathbf{X}(I)}:=\sum_{j=1}^{l}\|u_{j}\|_{X(I)}\leq a \right\}.$$
	In fact, for $T=T(\rho)\approx \rho^{\frac{4(p-1)}{n(p-1)-4}}$, $\Gamma:B(T,a)\to B(T,a)$ is well defined and  a contraction.\\
	
	The proof for the critical $L^2$ case, $p=1+\frac{4}{n}$ follows a similar idea, but we have to apply the contraction mapping principle on the closed  ball	
$$
\tilde{B}(T,a)=\left\{\mathbf{u}\in \mathbf{{L}}^{p+1}(I;\mathbf{L}^{p+1}):\|\mathbf{u}\|_{ \mathbf{{L}}^{p+1}(I;\mathbf{L}^{p+1})}:=\sum_{j=1}^{l}\|u_{j}\|_{L^{p+1}(I;L^{p+1})}\leq a \right\}.
$$

Finally, with these facts, we have the following results related to the existence of local $L^2$-solutions in the subcritical and critical cases.

 \begin{teore}\label{localexistenceL2} Let $1< p< 1+\frac{4}{n}$. Assume that   \textnormal{\ref{H1}} and \textnormal{\ref{H2*}} hold. Then for any $\rho>0$ there exists $T=T(\rho)>0$ such that for any $\mathbf{u}_0\in \mathbf{L}^2 $ with $\|\mathbf{u}_0\|_{\mathbf{L}^2}\leq \rho$, system \eqref{system1}
 has a unique   solution $\mathbf{u}\in \mathbf{X}(I)$ with $I=[-T,T]$.
\end{teore}

\begin{teore}\label{locexistL2n=4}
    Let $p=1+\frac{4}{n}$. Assume that  \textnormal{\ref{H1}} and \textnormal{\ref{H2*}} hold. Then for any $\mathbf{u}_0\in \mathbf{L}^2$, there exists $T(\mathbf{u}_0)>0$ (depending on $\mathbf{u}_0$) such that  system \eqref{system1}
 has a unique  solution $\mathbf{u}\in\mathbf{X}(I)$ with $I=[-T(\mathbf{u}_0),T(\mathbf{u}_0)]$.
\end{teore}

 Now, we make some comments about the local well-posedness theory in $H^1$. We consider  the following spaces



\begin{equation}\label{defq}
Y(I)=\begin{cases}
(\mathcal{C}\cap L^{\infty})(I; H^1)\cap L^{q}\left(I;W^{1}_{p+1}\right),\qquad \mbox{$p$ satisfying \eqref{pcondt}},\quad q=\frac{4(p+1)}{n(p-1)};\\
(\mathcal{C}\cap L^{\infty})(I; H^1)\cap L^{2}(I;W^{1}_{p+1}),\qquad p=\frac{n+2}{n-2}, \qquad n\geq 3,
\end{cases}
\end{equation}
on the time interval $I=[-T,T]$ with $T>0$. The norm in $Y(I)$ is defined as
\begin{equation*}
\|f\|_{Y(I)}=\begin{cases}
\|f\|_{L^{\infty}( H^1)}+ \|f\|_{L^{q}( W^{1}_{p+1})},\qquad \mbox{$p$ satisfying \eqref{pcondt}},\quad q=\frac{4(p+1)}{n(p-1)};\\
\|f\|_{L^{\infty}( H^1)}+ \|f\|_{L^{2}( W^{1}_{p+1})}, \qquad p=\frac{n+2}{n-2}, \qquad n\geq 3.
\end{cases}
\end{equation*}
 and define 
 \begin{equation}\label{thetaY}
 \theta(n,p)=
 \begin{cases}
 \qquad 1,\qquad\qquad\quad\mbox{if}\quad n=1;\\
 \frac{2+n+(2-n)p}{2(p+1)}, \quad\mbox{if}\quad  n\geq 2.
 \end{cases}
 \end{equation}

From the definition of spaces  $X(I)$ and $Y(I)$, we conclude that $Y(I)\hookrightarrow X(I)$ and $\|\nabla f\|_{X(I)}\leq  \|f\|_{Y(I)}$. This combined with H\"{o}lder inequalities and the embedding $H^{1}(\R)\hookrightarrow L^{\infty}(\R)$ gives for $f,g\in Y(I)$

\begin{equation}\label{estfg1}
\|f^{p-1}g\|_{L^{1}( L^{2})}\leq
CT^{\theta(n,p)}\|f\|^{p-1}_{Y(I)}\|g\|_{Y(I)}, \quad\mbox{if}\quad  n=1.
\end{equation}

Also, using H\"older inequality and the embeddings, $H^{1}(\R^{n})\hookrightarrow L^{p+1}(\R^{n})$ and $W^{1}_{p+1}(\R^{n})\hookrightarrow L^{p+1}(\R^{n})$ we get for $f,g\in Y(I)$

\begin{equation}\label{estfg2}
 \|f^{p-1}g\|_{L^{q'}(L^{r'})}\leq C T^{\theta(n,p)}\|f\|^{p-1}_{Y(I)}\|g\|_{Y(I)}, \quad if\quad n\geq 2.
 \end{equation}

As a consequence of \eqref{estfg1} and \eqref{estfg2}  Lemma \ref{limfk} we have  for
$\mathbf{u}, \mathbf{u}'\in \mathbf{Y}(I)$, 
 \begin{equation*}
\left\|\int_{0}^{t}U_{k}(t-t') \frac{1}{\alpha_{k}}[ f_{k}(\mathbf{u})-f_{k}(\mathbf{u}')]\;dt'\right\|_{Y(I)}\leq CT^{\theta(n,p)}\left(\|\mathbf{u}\|^{p-1}_{\mathbf{Y}(I)}+\|\mathbf{u}'\|^{p-1}_{\mathbf{Y}(I)}\right)\|\mathbf{u}-\mathbf{u}'\|_{\mathbf{Y}(I)},
\end{equation*}
 for some time interval $I$ and $\theta(n,p)$ given by \eqref{thetaY}.

 From the above discussion, a  contraction mapping argument the following  local $H^1$-solutions existence results in the subcritical and critical cases.

 \begin{teore}\label{localexistenceH1} Let $p$ satisfy \eqref{pcondt}. Assume that \textnormal{\ref{H1}} and \textnormal{\ref{H2*}} hold. Then for any $r>0$ there exists $T(r)>0$ such that for any $\mathbf{u}_0\in \mathbf{H}^1 $ with $\|\mathbf{u}_0\|_{\mathbf{H}^1}\leq r$,  system \eqref{system1}
 has a unique  solution $\mathbf{u}\in \mathbf{Y}(I)$ with $I=[-T(r),T(r)]$.
\end{teore}

\begin{teore}\label{locexistH1n=6} Let $p=\frac{n+2}{n-2}$, $n\geq 3$. Assume that  \textnormal{\ref{H1}} and \textnormal{\ref{H2*}} hold. Then for any  $\mathbf{u}_0:=(u_{10},\ldots,u_{l0})\in \mathbf{H}^1 $ there exists $T(\mathbf{u}_0)>0$ such that  system \eqref{system1}
 has a unique  solution $\mathbf{u}=(u_{1},\ldots,u_{l})\in\mathbf{Y}(I)$ with $I=[-T(\mathbf{u}_0),T(\mathbf{u}_0)]$.
\end{teore}

 Finally, as a consequence  of Theorems \ref{localexistenceL2} and  \ref{localexistenceH1},  we have the following blow up alternative:  there exist $T_*,T^*\in(0,\infty]$ such that the local solutions can be extend to the interval $(-T_*,T^*)$; moreover if $T_*<\infty$  (respect. $T^*<\infty$), then
$$
\lim_{t\to -T_*}\|\mathbf{u}(t)\|_{\mathbf{L}^2}=\infty, \qquad (respect.\lim_{t\to T^*}\|\mathbf{u}(t)\|_{\mathbf{L}^2}=\infty  ),
$$
for $L^2$-solutions, and 
$$
\lim_{t\to -T_*}\|\mathbf{u}(t)\|_{\mathbf{H}^1}=\infty, \qquad (respect.\lim_{t\to T^*}\|\mathbf{u}(t)\|_{\mathbf{H}^1}=\infty  ),
$$
for $H^1$-solutions.

 \section{Existence of ground states solutions}\label{sec.gs}

 This section is devoted to  proving the existence of standing wave  solutions for \eqref{system1} for the subcritical regime \eqref{pcondt}. This is achieved by considering the related nonlinear elliptical system and a special solution of it called \textit{ground state} solution. To do so, we will consider  minimizers of a Weinstain-type functional and use the fact that the space of radially-symmetric non-increasing functions, $H_{rd}^{1}(\R^{n})$,  is compactly embedded in  $L^{p+1}(\R^{n})$, for $p$ satisfying \eqref{pcondt}.  Throughout this section we assume that H's assumptions hold. In fact, assumption \ref{H8} is fundamental to be able to take radially symmetric minimizing sequences of the  Weinstain-type functional (see Lemma \ref{lemma2} (iv)). \\

The \textit{ground states} allows us to find an explicit expression for the best constant of a Gagliardo-Nirenberg type inequality, since this inequality is related to the Weinstain-type functional.  Let  $p$ satisfy the subcritical regime \eqref{pcondt}. From  the definition of the functionals $P$, $Q$ and $K$ (see \eqref{conservationcharge} and \eqref{funcK1}),   Lemma \ref{estdifF} i)  and assumption \ref{H6} we have for $\mathbf{u} \in  \mathbf{H}^{1}$

\begin{equation*}
\begin{split}
|P(\mathbf{u})|
&\leq \int \left|\mathrm{Re}\,F(\mathbf{u})\right|\;dx\leq C \sum_{k=1}^{l}\|u_{k}\|_{L^{p+1}}^{p+1}\leq  C_0Q(\mathbf{u})^{\frac{p-1}{2}(1-s_c)}K(\mathbf{u})^{\frac{p-1}{2}s_c+1},
\end{split}
\end{equation*}
where $C_0$ is a positive constant depending on $\alpha_{k}$ and $\gamma_{k}$, for $k=1,\ldots,l$ and $s_{c}$ is the critical exponent defined in \eqref{critexp}.

Thus, we obtain the following vectorial Gagliardo-Nirenberg-type inequality:
\begin{equation}\label{GNE2}
|P(\mathbf{u})| \leq C_{opt}Q(\mathbf{u})^{\frac{p-1}{2}(1-s_c)}K(\mathbf{u})^{\frac{p-1}{2}s_c+1}.
\end{equation}

In the final of this section we  obtain the best constant $C_{opt}$ in the inequality \eqref{GNE2} whose expression is in terms of the . \\

We start by replacing \eqref{standing} (standing wave solution) in   system  \eqref{system1}. From Lemma  \ref{H34impGC} ii) we know that for $k=1,\ldots,l$ and any $\omega\in \R$, we have
\begin{equation*}
f_{k}\left(e^{i\frac{\sigma_{1}}{2}\omega t}\psi_{1},\ldots,e^{\frac{\sigma_{l}}{2}\omega it}\psi_{l}\right)=e^{i\frac{\sigma_{k}}{2}\omega t}f_{k}(\psib),
\end{equation*}
where $\psib=(\psi_{1},\ldots,\psi_{l})$.

 Using this fact, we conclude that $\psi_{k}$ must satisfy the following elliptic system
\begin{equation}\label{systemelip}
\displaystyle -\gamma_{k}\Delta \psi_{k}+\left(\frac{\sigma_{k}\alpha_{k}}{2}\omega+\beta_{k}\right) \psi_{k}=f_{k}(\psib),\qquad k=1,\ldots,l.
\end{equation}

Note that by  Lemma \ref{fkreal}),  the nonlinearities $f_{k}$ are real-valued functions, so system \eqref{systemelip} is well defined. Also, we restrict our attention to values of $\omega$ such that $\displaystyle \omega  > -\frac{2\beta_{k}}{\sigma_{k}\alpha_{k}}$, since we are looking  for non-trivial solutions. \\

Next, we recall, \eqref{FunctionalI1}-\eqref{functionalQ}, that  the action functional associated with \eqref{systemelip}   is
\begin{equation}\label{FunctionalI2}
I(\psib)=\frac{1}{2}\left[K(\psib)+\mathcal{Q}(\psib)\right]-P(\psib).
\end{equation}
Also, we introduce the functional 
\begin{equation}\label{functionalJ}
J(\psib):=\frac{\mathcal{Q}(\psib)^{\frac{p-1}{2}(1-s_c)}K(\psib)^{\frac{p-1}{2}s_c+1}}{P(\psib)},\quad \quad P(\psib)\neq 0.
\end{equation}

Here are some comments about the functional defined above.

\begin{obs}
 It is not difficult to see that the functionals $K$, $\mathcal{Q}$, and $P$ are continuous on $ \mathbf{H}^1$ the continuity of  $P$ is a consequence of Lemma \ref{estdifF}). Also, these functionals have Fr\'echet derivatives. In particular, the action functional $I$ has Fr\'echet derivative.  In fact, the critical points of $I$ are the solutions of \eqref{systemelip}
\end{obs}

\begin{defi}\label{defgroundstate} Let $\mathcal{C}$  be the set of non-trivial critical points of $I$. We say that $\psib\in \mathbf{H}^1$ is a ground state solution of \eqref{systemelip} if 
\begin{equation*}
I(\psib)=\inf\left\{I(\boldsymbol{\phi}); \boldsymbol{\phi}\in \mathcal{C}\right\}. 
\end{equation*}
We denote by $\mathcal{G}(\omega,\boldsymbol{\beta})$ the set of all ground states for system \eqref{systemelip}, where  $(\omega,\boldsymbol{\beta})$ indicates the dependence on the parameters $\omega$ and $\boldsymbol{\beta}$.
\end{defi}

Now we establish some relations between the functionals $K,\mathcal{Q}, P$ and $I$. Some of them are similar to the well known Pohozaev's identities for elliptic equations.  
\begin{lem}
\label{identitiesfunctionals}
Let $p$ satisfy \eqref{pcondt}.
If $\psib$ is a non-trivial solution of \eqref{systemelip} then,
\begin{equation}
P(\psib)=\frac{2}{p-1}I(\psib),\label{b}\\
\end{equation}
\begin{equation}
K(\psib)=nI(\psib),\label{d}\\
\end{equation}
\begin{equation}
\mathcal{Q}(\psib)=2(1-s_c)I(\psib).\label{e}
\end{equation}
\begin{eqnarray}\label{relationJandI}
J(\psib)=\frac{(p-1)n^{\frac{p-1}{2}s_c+1}}{2}\left[2(1-s_c)\right]^{\frac{p-1}{2}(1-s_c)}I(\psib)^{\frac{p-1}{2}}.
\end{eqnarray}
In particular, any non-trivial solution $\psib \in \mathcal{P}$ of \eqref{systemelip}  which is a minimizer of $J$ is a ground state of \eqref{systemelip}.

\end{lem}
\begin{proof}
The proof can be achieved by adapting the proof in reference   \cite[Lemmas 4.4 and 4.7]{NoPa2}. 
\end{proof}

Before continuing, we introduce some useful notation. Given any non-negative function $f\in H^1(\mathbb{R}^n)$ we denote by $f^*$ its symmetric-decreasing rearrangement (see, for instance, \cite{Leoni} or \cite{Lieb}). Also, for any $\lambda>0$, $(\delta_{\lambda}g)(x)=g\left(\frac{x}{\lambda}\right)$. Thus, if $\mathbf{g}=(g_1,\ldots,g_l)\in \mathbf{H}^1$, we set $\mathbf{g}^*=(g_1^*,\ldots,g_l^*)$ and $(\delta_{\lambda}\mathbf{g})(x)=\left(g_1\left(\frac{x}{\lambda}\right), \ldots, g_l\left(\frac{x}{\lambda}\right)\right)$.

The functionals introduced in this section satisfy certain  properties about scaling transformations and symmetric-decreasing rearrangement. These properties can be translated into the following results for the functional $J$.

\begin{lem}\label{lemma2}
Let $n\geq 1$ and $a,\lambda>0$. If  $\psib\in \mathcal{P}$ and $\mathbf{g}\in (\mathcal{C}_{0}^{\infty}(\R^{n}))^{l}$ we have
\begin{enumerate}
\item[(i)] $J(a\delta_{\lambda}\psib)=J(\psib)$;
\item[(ii)] $J(|\psib|)\leq J(\psib)$, where $|\psib|=(|\psi_{1}|,\ldots,|\psi_{l}|)$;
    \item[(iii)] $J'(a\delta_{\lambda}\psib)=a^{-1}J'(\psib)(\delta_{\lambda^{-1}}\mathbf{g})$.\\
    In addition, if $\psi_{k}$ is non-negative, for $k=1,\ldots,l$, then
    \item[(iv)] $J(\psib^{*})\leq J(\psib)$.
\end{enumerate}
\end{lem}

\begin{proof}
The  proof of (i) and (iii)  are immediate consequences of the definitions of functionals $\mathcal{Q}$, $K$ and $P$ . For (ii) we must  use assumption \ref{H5*}. To prove (iv) we need  to use that
\begin{equation*}
\|\psi^{*}\|_{L^{2}}=\|\psi\|_{L^{2}}, \qquad \mbox{and} \qquad\|\nabla\psi^{*}\|_{L^{2}}\leq\|\nabla \psi\|_{L^{2}}. 
\end{equation*}
\end{proof}
and it is fundamental to use assumption \ref{H8}, to conclude $P(\psib^*)\geq P(\psib)$.\\

With the above lemmas in hand, we are able to present our main result concerning ground states. As usual, we will say that a function $\psib\in\mathbf{H}^1$ is positive (non-negative), and write $\psib>0$ ($\psib\geq0$), if each one of its components are positive (non-negative). Also, $\psib$ is radially symmetric if each one of its components are radially symmetric.\\

We are now in position to prove the existence  of \textit{ground states} solutions to system \eqref{system1}. This construction  is  adapted from  \cite{Weinstein}, in  which the result   for the single Nonlinear Schr\"{o}dinger equation is demonstrated (see also \cite{NoPa2}).  
\begin{teore}[Existence of ground state solutions]\label{thm:existenceGSJgeral}
 Assume that  \textnormal{\ref{H1}-\ref{H8}} hold. Let $p$ satisfy \eqref{pcondt} and  $ s_c= \frac{n}{2}-\frac{2}{p-1}$. The infimum
\begin{equation}\label{xi1}
\xi_1=\inf\limits_{\psib\in \mathcal{P}}J(\psib), \quad \mathcal{P}:=\{\psib\in \mathbf{H}^{1};\, P(\psib)>0\}
\end{equation}
 is attained at a function $\psib_0\in \mathcal{P} $ with the following properties:
\begin{enumerate}
\item[(i)] $\psib_0$ is a non-negative and radially symmetric function.
\item[(ii)]   Up to scaling $\psib_0$ is a positive ground state solution of \eqref{systemelip}.
In addition,
if $\tilde{\psib}$ is any ground state of \eqref{systemelip} then the infimun $\xi_1$ can be characterized by mean of the $L^2$ norm of $\tilde{\psib}$ as follow
\begin{equation}\label{inffunctionalJ}
\xi_{1}=\frac{(p-1)n^{\frac{p-1}{2}s_c+1}}{2}\left[2(1-s_c)\right]^{\frac{p-1}{2}(1-s_c)}\mathcal{Q}(\psib)^{\frac{p-1}{2}}.
\end{equation}
\end{enumerate}
\end{teore}
\begin{proof}
Let $(\psib_j)\subset \mathcal{P} $ be a minimizing sequence for \eqref{xi1}, i.e., 
$$\lim_{j\to \infty}J(\psib_j)=\xi_{1}.$$
Replacing $\psib_j$ by $|\psib_j|^*$, from Lemma \ref{lemma2}   we  may assume that $\psib_j$ are radially symmetric and non-increasing  functions in $\mathbf{H}^{1}$. Define $\tilde{\psib}_{j}=t_{j}\delta_{\lambda_{j}}\psib_{j}$,  where
$$t_{j}=\frac{\mathcal{Q}(\psib_j)^{\frac{p-1}{2}(s_c-1)+1}}{K(\psib_j)^{\frac{p-1}{2}s_c+1}}\qquad\mbox{and}\qquad \lambda_{j}=\frac{K(\psib_j)^{\frac{p-1}{2}}}{\mathcal{Q}(\psib_j)^{\frac{p-1}{2}}}.$$
From \eqref{funcK1}, \eqref{functionalQ} and   Lemma  \ref{lemma2}  with $a=t_{j}$ and $\lambda=\lambda _{j}$ we have
\begin{equation}\label{KQlimitado}
K(\tilde{\psib}_{j})=\mathcal{Q}(\tilde{\psib}_{j})=1 \quad \mbox{and} \quad J(\tilde{\psib}_j)=J(\psib_j).
\end{equation}
Hence,
\begin{equation}\label{convergenceP3}
\begin{split}
\frac{1}{P(\tilde{\psib}_{j})}=J(\tilde{\psib}_{j})=J(\psib_j)\to \xi_{1}>0.
\end{split}
\end{equation}
In view of \eqref{KQlimitado}, the sequence $(\tilde{\psib}_j)$ is bounded in $\mathbf{H}_{rd}^{1}$. By recalling that the embedding $H^{1}_{rd}(\R^{n}) \hookrightarrow L^{p+1}(\R^{n})$ is compact for $p$ in the subcritical regime \eqref{pcondt} (see Proposition 1.7.1 in \cite{Cazenave}), there exist a subsequence, still denoted by $(\tilde{\psib}_{j})$, and $\psib_0\in \mathbf{H}_{rd}^{1}$  such that
\begin{equation}\label{strongcnvergenceL3Rn2}
\begin{cases}
\tilde{\psib}_j\rightharpoonup \psib_{0}, \qquad\mbox{in \quad $\mathbf{H}^{1}$},\\
\tilde{\psib}_{j}\to \psib_{0}\qquad\mbox{in\quad $\mathbf{L}^{p+1}$},\\
\tilde{\psib}_{j}\to \psib_{0}\qquad\mbox{ a.e \quad in\quad $\R^n$}.
\end{cases}
 \end{equation}
The last convergence in \eqref{strongcnvergenceL3Rn2} implies that $\psib_{0}$ is non-negative and radially symmetric. In addition, since by  Lemma \ref{estdifF},
\begin{equation*}
\begin{split}
\left|P(\tilde{\psib}_{j})-P(\psib_{0})\right|
&\leq \int\left|F(\tilde{\psib}_{j})-F(\psib_{0})\right|\;dx\\
&\leq C \sum_{m=1}^{l}\sum_{k=1}^{l}\int(|\tilde{\psi}_{kj}|^{p}+|\psi_{k0}|^{p})|\tilde{\psi}_{mj}-\psi_{m0}|\;dx \\
&\leq C \sum_{m=1}^{l}\sum_{k=1}^{l}(\|\tilde{\psi}_{kj}\|^{p}_{L^{p+1}}+\|\psi_{k0}\|^{p}_{L^{p+1}})\|\tilde{\psi}_{mj}-\psi_{m0}\|_{L^{p+1}},
\end{split}
\end{equation*}
we deduce from \eqref{strongcnvergenceL3Rn2} and (\ref{convergenceP3}) that
\begin{equation}\label{relationPandalpha2}
P(\psib_{0})=\lim_{j\to \infty}P(\tilde{\psib}_{j})=\xi_{1}^{-1}>0,
\end{equation}
which means that $\psib_0\in \mathcal{P}$.

On the other hand, the lower semi-continuity of the weak convergence  gives 
\begin{equation*}
K(\psib_{0})\leq \liminf_{j}K(\tilde{\psib}_{j})=1\quad
\mbox{and}\quad
\mathcal{Q}(\psib_{0})
\leq\liminf_{j}\mathcal{Q}(\tilde{\psib}_{j})=1.
\end{equation*}
Therefore, (\ref{relationPandalpha2}) yields
\begin{equation}\label{inequJKandQ2}
\xi_{1}\leq  J(\psib_{0})=\frac{\mathcal{Q}(\psib_{0})^{\frac{p-1}{2}(1-s_c)}K(\psib_{0})^{\frac{p-1}{2}s_c+1}}{P(\psib_{0})}\leq \frac{1}{P(\psib_{0})}=\xi_{1}.
\end{equation}
From \eqref{inequJKandQ2} we conclude that 
\begin{equation*}
    J(\psib_{0})=\xi_{1}
\end{equation*}
and
\begin{equation}\label{equal1}
    K(\psib_{0})=\mathcal{Q}(\psib_{0})=1.
\end{equation}
A combination of the last assertion with \eqref{strongcnvergenceL3Rn2} also implies that
 $\tilde{\psib}_{j}\to \psib_{0}$ strongly in $\mathbf{H}^{1}$. Part (i) of the theorem is thus established.
 
  For  part (ii) we  note that for $t$ sufficiently small and $\mathbf{u}\in\mathbf{H}^1$, $(\boldsymbol{\psi}_{0}+t\mathbf{u})\in \mathcal{P}$. Thus, since $\boldsymbol{\psi}_{0} $ is a minimizer of $J$ on $\mathcal{P} $  we have  
\begin{equation*}
\left.\frac{d}{dt}\right|_{t=0}J(\boldsymbol{\psi}_{0}+t\mathbf{u})=0,
\end{equation*}
which in view of the definions of functionals $K$, $\mathcal{Q}$ and $P$ is equivalent to
\begin{multline*}
\frac{\mathcal{Q}(\boldsymbol{\psi}_{0})^{\frac{p-1}{2}(1-s_c)}K(\boldsymbol{\psi}_{0})^{\frac{p-1}{2}s_c+1}}{P(\boldsymbol{\psi}_{0})}\left\{\left(\frac{p-1}{2}s_c+1\right)\frac{K'(\boldsymbol{\psi}_{0})(\mathbf{u})}{K(\boldsymbol{\psi}_{0})}+\left[\frac{p-1}{2}(1-s_c)\right]\frac{\mathcal{Q}'(\boldsymbol{\psi}_{0})}{\mathcal{Q}(\boldsymbol{\psi}_{0})}\right\}\\
=\frac{\mathcal{Q}(\boldsymbol{\psi}_{0})^{\frac{p-1}{2}(1-s_c)}K(\boldsymbol{\psi}_{0})^{\frac{p-1}{2}s_c+1}}{P(\boldsymbol{\psi}_{0})^{2}}P'(\boldsymbol{\psi}_{0})(\mathbf{u}).
\end{multline*}
From \eqref{relationPandalpha2} and \eqref{equal1}  this yields
\begin{equation}
K'(\boldsymbol{\psi}_{0})(\mathbf{u})+\frac{2}{n}(1-s_c)\mathcal{Q}'(\boldsymbol{\psi}_{0})(\mathbf{u})=\frac{4\xi_{1 }}{n(p-1)}P'(\boldsymbol{\psi}_{0})(\mathbf{u}).\label{relationKQandPgeral}
\end{equation}
Next, define  $\psib=t_{0}\delta_{\lambda_{0}}\psib_{0}$ with 
$$t_{0}=\displaystyle \left[\frac{2 \xi_{1} }{2(p+1)-n(p-1)}\right]^{1/(p-1)}\qquad\mbox{and}\qquad \lambda_{0}=\left[\frac{2}{ n}(1-s_c)\right]^{1/2}.$$
 We claim that $\psib$ is a solution of (\ref{systemelip}). Indeed,   for any  $\mathbf{u}\in \mathbf{H}^{1}$ in view of  \eqref{relationKQandPgeral},
\begin{equation*}
\begin{split}
\quad I'(\boldsymbol{\psi})(\boldsymbol{u})
&=\frac{1}{2}\left[K'(\boldsymbol{\psi})(\boldsymbol{u})+\mathcal{Q}'(\boldsymbol{\psi})(\boldsymbol{u})\right]-P'(\boldsymbol{\psi})(\boldsymbol{u})\\
&=\frac{1}{2}\left[K'(t_{0}\delta_{\lambda_{0}}\boldsymbol{\psi}_{0})(\boldsymbol{u})+ \mathcal{Q}'(t_{0}\delta_{\lambda_{0}}\boldsymbol{\psi}_{0})(\boldsymbol{u})\right]-P'(t_{0}\delta_{\lambda_{0}}\boldsymbol{\psi}_{0})(\boldsymbol{u})\\
&=\frac{t_{0}}{2}\left[K'(\delta_{\lambda_{0}}\boldsymbol{\psi}_{0})(\boldsymbol{u})+ \mathcal{Q}'(\delta_{\lambda_{0}}\boldsymbol{\psi}_{0})(\boldsymbol{u})\right]-t_{0}^{p}P'(\delta_{\lambda_{0}}\boldsymbol{\psi}_{0})(\boldsymbol{u})\\
&=\frac{t_{0}}{2}\left[\lambda_{0}^{n-2}K'(\boldsymbol{\psi}_{0})(\delta_{\lambda_{0}^{-1}}\boldsymbol{u})+\lambda_{0}^{n}\mathcal{Q}'(\boldsymbol{\psi}_{0})(\delta_{\lambda_{0}^{-1}}\boldsymbol{u})\right]-t_{0}^{p}\lambda_{0}^{n}P'(\boldsymbol{\psi}_{0})(\delta_{\lambda_{0}^{-1}}\boldsymbol{u})\\
&=\frac{t_{0}\lambda_{0}^{n-2}}{2}\left[K'(\boldsymbol{\psi}_{0})(\delta_{\lambda_{0}^{-1}}\boldsymbol{u})+\lambda_{0}^{2} \mathcal{Q}'(\boldsymbol{\psi}_{0})(\delta_{\lambda_{0}^{-1}}\boldsymbol{u})-2t_{0}^{p-1}\lambda_{0}^{2}P'(\boldsymbol{\psi}_{0})(\delta_{\lambda_{0}^{-1}}\boldsymbol{u})\right]\\
&=\frac{t_{0}\lambda_{0}^{n-2}}{2}\left[K'(\boldsymbol{\psi}_{0})(\delta_{\lambda_{0}^{-1}}\boldsymbol{u})+\frac{2(p+1)-n(p-1)}{n(p-1)}\mathcal{Q}'(\boldsymbol{\psi}_{0})(\delta_{\lambda_{0}^{-1}}\boldsymbol{u})\right.\\
&\quad-\left.\frac{4\xi_{1}}{n(p-1)}P'(\boldsymbol{\psi}_{0})(\delta_{\lambda_{0}^{-1}}\boldsymbol{u})\right]\\
&=0.
\end{split}
\end{equation*}

Now from Lemmas \ref{identitiesfunctionals} and \ref{lemma2}, we have that $\psib$ is also a critical point of $J$ with $J(\psib)=J(\psib_{0})$. Since $\psib_{0}$ is a minimizer of $J$, so is $\psib$. Another application of Lemma \ref{identitiesfunctionals} gives that $\psib$ is a ground state of (\ref{systemelip}). To see that $\psib$ is positive, we note that
\begin{equation*}
\Delta \psi_{k} -\frac{b_{k}}{\gamma_{k}} \psi_{k}=-\frac{1}{\gamma_{k}}f_{k}(\psib)\leq 0,
\end{equation*}
because $\gamma_{k}>0$,  $\psi_{k}$ are non-negatives  and $f_{k}$ satisfy \ref{H7}. Therefore by the strong maximum principle (see, for instance, \cite[Theorem 3.5]{gil})  we obtain the positiveness of $\psib$.
 
Finally, we will prove   (\ref{inffunctionalJ}). Indeed, if $\psib$ is as in part (ii), Lemma \ref{identitiesfunctionals} implies,
\begin{equation*}
\begin{split}
    \xi_{1}=J(\psib)
&=\frac{(p-1)n^{\frac{p-1}{2}s_c+1}}{2}\left[2(1-s_c)\right]^{\frac{p-1}{2}(1-s_c)}I(\psib)^{\frac{p-1}{2}}\\
&=\frac{(p-1)n^{\frac{p-1}{2}s_c+1}}{2}\left[2(1-s_c)\right]^{\frac{(1-p)}{2}s_c}\mathcal{Q}(\psib)^{\frac{p-1}{2}}.
\end{split}
\end{equation*}
Therefore, if $\tilde{\psib}\in\mathcal{G}(\omega,\boldsymbol{\beta})$, we get
$$\xi_{1}=\frac{(p-1)n^{\frac{p-1}{2}s_c+1}}{2}\left[2(1-s_c)\right]^{\frac{(1-p)}{2}s_c}\mathcal{Q}(\tilde{\psib})^{\frac{p-1}{2}},$$
completing the proof of the theorem.
\end{proof}

As a consequence of Theorem \ref{thm:existenceGSJgeral} we can obtain the optimal constant in the Gagliardo-Nirenberg-type inequality \eqref{GNE2}. 

\begin{coro}\label{corollarybestconstant}
Let $p$ satisfy \eqref{pcondt} and $s_{c}$ be defined as \eqref{critexp}.  The inequality 
\begin{equation*}
P(\mathbf{u})\leq C_{opt}\mathcal{Q}(\mathbf{u})^{\frac{p-1}{2}(1-s_c)}K(\mathbf{u})^{{\frac{p-1}{2}s_c+1}},
\end{equation*}
holds, for any $\mathbf{u}\in \mathcal{P}$, with
\begin{equation*}
C_{opt}=\frac{2}{(p-1)n^{\frac{p-1}{2}s_c+1}}\left[2(1-s_c)\right]^{\frac{(p-1)}{2}s_c}\frac{1}{\mathcal{Q}(\psib)^{\frac{p-1}{2}}},
\end{equation*}
where $\psib \in \mathcal{G}(\omega,\boldsymbol{\beta})$.
\end{coro}

 \section{Proof of Theorems \ref{thm:globalexistencecondn=51} and \ref{thm:globalexistencecondn=52}}\label{sec.gsbu}
 In this section we  give a proof of the dichotomy global solutions versus blow-up in finite time presented in Theorems \ref{thm:globalexistencecondn=51} and \ref{thm:globalexistencecondn=52}.  The sharp criterion for global well-posedness will be given in terms of such ground states.  

  \subsection{Conservation of the charge and energy}

  From the alternative blow-up results presented at the final of Theorem \ref{locexistH1n=6} we see that in order to   to get global solutions we need to find an \textit{a priori}  estimate for the  $L^{2}$ and  $H^{1}$ norms of local solution. Thus, the next step  is to  extend globally-in-time the solutions given by Theorems \ref{localexistenceL2} and  \ref{localexistenceH1}. A common strategy to do so is using the conservation of mass and energy. Then, in the following,  we will refer to this property. \\

 The conservation of the mass defined in \eqref{conservationcharge} is a consequence of conditions \textnormal{\ref{H3}} and \textnormal{\ref{H4}}. We multiply \eqref{system1}  by $\overline{u}_{k}$, where $k=1,\ldots, l$, then integrating by parts in $x$ and taking the imaginary part of the result we get after summing over $k$
\begin{equation*}
\frac{d}{dt}\left(\sum_{k=1}^{l}\frac{\sigma_{k}\alpha_{k}}{2}\|  u_{k}(t)\|_{L^{2}}^{2}\right)=-2\mathrm{Im}\int \sum_{k=1}^{l}\sigma_{k}f_{k}(\ub)\overline{u}_{k}\;dx=0,
\end{equation*}
 where we have used in \textnormal{\ref{H4}} in the last equality. Thus, the mass of    system \eqref{system1} is a conserved quantity.\\

An immediate consequence of the conservation of the mass is the extension of the local solutions in time in $L^2$. More precisely, we have 

\begin{teore}\label{L2normconserved}
	Let $1< p< 1+\frac{4}{n}$. Assume that   \textnormal{\ref{H1}-\ref{H4}} hold. Then for any $\mathbf{u}_0 \in \mathbf{L}^2$, the corresponding solution  $\mathbf{u}$ with maximal time interval of existence $I$, can be extended globally in time, that is $I=(-\infty,\infty)$. Moreover,
	\begin{equation*}
	Q(\mathbf{u}(t))=Q(\mathbf{u}_0),\qquad\forall t\in \R.
	\end{equation*}
\end{teore}

 Now, we discuss the conservation of the energy. From \textnormal{\ref{H3}} and the chain rule, we have

\begin{equation}\label{conservener1}
\begin{split}
\mathrm{Re}\left[\sum_{k=1}^{l}f_{k}\partial_{t}\overline{u}_{k}\right]&=\mathrm{Re}\left[\frac{d}{dt}F(\mathbf{u}(t))\right].
\end{split}
\end{equation}
Next, multiplying  \eqref{system1} by $\partial_t \overline{u}_k$, adding its complex conjugate and integrating by parts in $x$ we get after taking the sum over $k$
\begin{equation}\label{ener1}
\frac{d}{dt}\left(\sum_{k=1}^{l}\gamma_{k}\| \nabla u_{k}\|_{L^{2}}^{2}+\sum_{k=1}^{l}\beta_{k}\|  u_{k}\|_{L^{2}}^{2}\right)=2\int\mathrm{Re}\left[\sum_{k=1}^{l}  f_{k}(\ub)\partial_{t}\overline{u}_{k}\right]\;dx.
 \end{equation}
combining  \eqref{conservener1} and \eqref{ener1} we obtain 
\begin{equation*}
    \frac{d}{dt}\left(\sum_{k=1}^{l}\gamma_{k}\|\nabla u_{k}(t)\|_{L^2}^{2}+\sum_{k=1}^{l}\beta_{k}\|u_{k}(t)\|_{L^2}^{2}
    -2\mathrm{Re}\int F(\mathbf{u}(t))\;dx\right)=0.
\end{equation*}
Hence, the energy
given by 
\begin{equation*}
	E(\mathbf{u}(t))=\sum_{k=1}^{l}\gamma_{k}\|\nabla u_{k}(t)\|_{L^2}^{2}+\sum_{k=1}^{l}\beta_{k}\|u_{k}(t)\|_{L^2}^{2}
    -2\mathrm{Re}\int F(\mathbf{u}(t))\;dx,
\end{equation*}
 is  conserved.

 \subsection{Properties under mass resonance condition}\label{virialindet}
This section is devoted to stating  some useful virial-type identities satisfied by the solutions of system \eqref{system1} as well as an invariance by a pseudo-conformal-type transform of these solutions. These properties are established assuming the mass-resonance condition. 

We start discussing which result we gain by assuming \eqref{RC}.

\begin{obs}\label{obsmassreso}
    Considering assumption \textnormal{\ref{H3}} and \eqref{RC}, we can repeat the argument in the proof of Lemma \ref{H34impGC} and conclude that  for any $\theta \in \R$ and $\zb \in \C^{l}$ we have for $k=1,\ldots,l$

  \begin{itemize}
      \item[i)]  the invariance 
     \begin{equation*}
\mathrm{Re}\,F\left(e^{i\frac{\alpha_{1}}{\gamma_1}\theta  }z_{1},\ldots,e^{i\frac{\alpha_{l}}{\gamma_l}\theta  }z_{l}\right)=\mathrm{Re}\,F(\zb),
\end{equation*}
\item[ii)]
and the Gauge condition
     \begin{equation*}
    f_{k}\left(e^{i\frac{\alpha_{1}}{\gamma_1}\theta }z_{1},\ldots,e^{i\frac{\alpha_{l}}{\gamma_l}\theta }z_{l}\right)=e^{i\frac{\alpha_{k}}{\gamma_k}\theta }f_{k}(z_{1},\ldots,z_{l}).
\end{equation*}

  \end{itemize}

\end{obs}

\subsubsection{Virial Identities}

As we said, a very important tool for constructing blow-up solutions are the virial identities. To establish such identities we start with the variance function. The first result gives conditions under which the variance is differentiable. We will start by introducing the following space
\begin{equation*}
\Sigma=\{\mathbf{u}\in \mathbf{H}^{1}; \;x\mathbf{u}\in \mathbf{L}^{2}\}.
\end{equation*}
Here the product $x\mathbf{u}$ must be understood as $(xu_1,\ldots,xu_l)$. In particular,
$$
\|x\mathbf{u}\|_{\mathbf{L}^2}=\sum_{k=1}^l\int|x|^2|u_k|^2dx.
$$
We note that $\Sigma$ equipped with the norm
\begin{equation*}
\|\mathbf{u}\|_{\Sigma}=\|\mathbf{u}\|_{\mathbf{H}^{1}}+\||\cdot|\mathbf{u}\|_{\mathbf{L}^{2}},
\end{equation*}
is a Hilbert space.

\begin{teore}\label{persistL2wheit}  Assume  system \eqref{system1} satisfies the mass resonance condition \eqref{RC} and $\mathbf{u}_0\in \Sigma$. Let $\mathbf{u}$ be the corresponding  local solution given by Theorems \ref{localexistenceH1} and \ref{locexistH1n=6}. Then, the function $t\to |\cdot|\mathbf{u}(\cdot,t)$ belongs to $\mathcal{C}(I,\mathbf{L}^{2})$. Moreover, the function \begin{equation}
\label{fuctV}
 t\to V(t)=\sum_{k=1}^{l}\frac{\alpha_{k}^{2}}{\gamma_{k}}\|xu_{k}(t)\|_{L^2}^{2}=\sum_{k=1}^{l}\frac{\alpha_{k}^{2}}{\gamma_{k}}\int|x|^2|u_{k}(x,t)|^{2}\;dx
 \end{equation}
  is in $\mathcal{C}^{2}(I)$, 
 \begin{equation*}
 V'(t)=4\sum_{k=1}^{l}\alpha_{k}\mathrm{Im}\int\nabla u_{k}\cdot x\overline{u}_{k}\,dx,
 \end{equation*}
 and
 \begin{equation}\label{secderV}
 V''(t)=2 n(p-1)E(\mathbf{u}_0)-2 n (p-1) L(\mathbf{u})+2 (4-np+n)K(\mathbf{u}), \qquad \mbox{for all}\;\; t\in I,
 \end{equation}
 where $K$ and $L$ are defined in \eqref{funcK1}.
\end{teore}
\begin{proof}
For a single nonlinear equation	in  \cite[Proposition 6.5.1]{Cazenave},    a proof can be found. Following these ideas a  proof of Theorem \ref{persistL2wheit}  can be achieved. A different approach, employed in reference  \cite{Corcho}, which uses the Hamiltonian structure of the system can also  be used. In fact, in reference \cite{NoPa2} this technique  was adapted  for system \eqref{system1}, when $p=2$. We  point out that in order to get $V'$ we need to use Remark \ref{obsmassreso} i) (which as we said is a consequence of mass resonance \eqref{RC}) and to get $V''$  is fundamental assumption \ref{H5*}. 
\end{proof}

By integrating twice \eqref{secderV} with respect to $t$ we obtain the following.

\begin{coro}[Virial identity]\label{conservationlawweightedspace}
	 Assume  system \eqref{system1} satisfies the mass resonance condition \eqref{RC}. Let  $\mathbf{u}_0\in \mathbf{H}^1$ and  $\mathbf{u}\in \Sigma$ be the corresponding solution given by Theorems \ref{localexistenceH1}, \ref{locexistH1n=6} and \ref{persistL2wheit}. Then
		\begin{multline*}
	Q\left(x\mathbf{u}(t)\right)=Q\left(x\mathbf{u}_0\right)+P_{0}t+n(p-1)E_{0}t^{2}-2 n(p-1)\int_{0}^{t}(t-s)L(\mathbf{u}(s))\; ds\\
	+2(4-np+n)\int_{0}^{t}(t-s)K(\mathbf{u}(s))\;ds,
	\end{multline*}
for all $t\in I$, where
$$P_{0}=4\sum_{k=1}^{l}\alpha_{k}\mathrm{Im}\int\nabla u_{k0}\cdot x\overline{u}_{k0}\;dx.$$
\end{coro}

Now we obtain a version of the Theorem \ref{persistL2wheit} when the initial data is radially symmetric. In this version we also consider a smooth cut-off function instead of $|x|$.

\begin{teore}\label{Viarialidenityradialcase} Assume  system \eqref{system1} satisfies the mass resonance condition \eqref{RC}. Let $ \mathbf{u}_0 \in \mathbf{H}^{1}$ and   $\mathbf{u}$ be the corresponding solution given by Theorems \ref{localexistenceH1} and \ref{locexistH1n=6}. Assume $\varphi \in C^{\infty}_{0}(\R^{n})$ and define
\begin{equation*}
V(t)=\frac{1}{2}\int \varphi(x)\left(\sum_{k=1}^{l}\frac{\alpha_{k}^{2}}{\gamma_{k}}|u_{k}|^{2}\right)\;dx.
\end{equation*}
Then,
\begin{equation*}
V'(t)=\sum_{k=1}^{l}\alpha_{k}\mathrm{Im}\int\nabla \varphi\cdot \nabla u_{k} \overline{u}_{k}\;dx,
\end{equation*}
and
\begin{equation}\label{secondervgeralcase}
\begin{split}
V''(t)&=2\sum_{1\leq m,j\leq n}\mathrm{Re}\int\frac{\partial^{2}\varphi}{\partial x_{m}\partial x_{j}}\left[\sum_{k=1}^{l}\gamma_{k}\partial x_{j}\overline{u}_{k}\partial x_{m}u_{k}\right]dx\\
&\quad-\frac{1}{2}\int\Delta^{2}\varphi\left(\sum_{k=1}^{l}\gamma_{k}|u_{k}|^{2}\right)\;dx+(1-p)\mathrm{Re}\int\Delta\varphi F\left(\mathbf{u}\right)\;dx.
\end{split}
\end{equation}
\end{teore}
\begin{proof}
The proof follows the ideas presented in  \cite[Lemma 2.9]{Kavian}. See also reference \cite[Theorem 5.7]{NoPa}  where a version in the case $p=2$ was considered.
\end{proof}

\begin{coro}\label{corovirrad}

Under the assumptions of Theorem \ref{Viarialidenityradialcase}, if $\varphi$ and $\mathbf{u}_0$ are 
 radially symmetric functions,  we can write \eqref{secondervgeralcase} as
 \begin{equation}\label{secondervradialcase}
 \begin{split}
V''(t)&=2\int \varphi''\left(\sum_{k=1}^{l}\gamma_{k}|\nabla u_{k}|^{2}\right)dx-\frac{1}{2}\int\Delta^{2}\varphi\left(\sum_{k=1}^{l}\gamma_{k}|u_{k}|^{2}\right)dx\\
&\quad+(1-p)\mathrm{Re}\int\Delta\varphi\, F\left(\mathbf{u}\right)\;dx.
\end{split}
\end{equation}
 \end{coro}

\begin{proof}
The proof is a consequence of Theorem  \ref{Viarialidenityradialcase},  where it has been used that the solution $\mathbf{u}$ is  radially symmetric, when $\mathbf{u}_0$ is also radial. 
\end{proof}

\subsubsection{Pseudo-conformal transform}\label{subconformal}

In the  $L^{2}$ critical case, $\displaystyle p=1+\frac{4}{n}$ and under the mass-resonance condition system \eqref{system1} is  invariance of the  under the pseudo-conformal transformation. In what follows, $SL(2,\R)$ denotes the special linear group of degree 2.
\begin{lem}\label{pseudoconf}
Let  $p=1+\frac{4}{n}$ and   let $A=\left(\begin{array}{cc}
      a&b  \\
      c&d 
 \end{array}\right)\in SL(2,\R)$. Assume  system \eqref{system1} satisfies the mass resonance condition \eqref{RC} and  define $\mathbf{v}^A=(v_1^A, \ldots,v_l^A)$ by
 \begin{equation*}
     v_{k}^{A}(x,t)=(a+bt)^{\frac{2}{1-p}}e^{i\frac{\alpha_{k}}{\gamma_{k}}\frac{b|x|^{2}}{4(a+bt)}}u_{k}\left(\frac{x}{a+bt},\frac{c+dt}{a+bt}\right),\quad k=1,\ldots,l.
 \end{equation*}
If $\mathbf{u}$ is a solution of system \eqref{system1} with $\beta_{k}=0$, $k=1,\ldots,l$, so is $\mathbf{v}^A$.
\end{lem}
\begin{proof}
First note that a straightforward  calculation gives 
\begin{equation}\label{calculation}
     \begin{split}
         &\left[i\alpha_{k}\partial_{t}v_{k}^{A}+\gamma_{k}\Delta v_{k}^{A}\right](x,t)\\
         &\qquad=(a+bt)^{\frac{2p}{1-p}}e^{i\frac{\alpha_{k}}{\gamma_{k}}\frac{b|x|^{2}}{4(a+bt)}}\left[i\alpha_{k}\partial_{t}u_{k}+\gamma_{k}\Delta u_{k}\right]\left(\frac{x}{a+bt},\frac{c+dt}{a+bt}\right),
     \end{split}
 \end{equation}
 for $ k=1,\ldots,l.$. For simplicity in the following we will drop the argument $\left(\frac{x}{a+bt},\frac{c+dt}{a+bt}\right)$ in the function $\mathbf{u}$. \\
 
 Since we are under the mass-resonance assumption, the nonlinearities $f_{k}$ satisfy the Gauge condition stated in Remark \ref{obsmassreso} (ii). Also, from Lemma \ref{estdifF} (iv) $f_{k}$ is homogeneous of degree $p$, then joining this with \eqref{calculation} we have for $ k=1,\ldots,l$
 
 \begin{equation*}
     \begin{split}
         f_{k}(\mathbf{v}^A(x,t))&=f_{k}\left((a+bt)^{\frac{2}{1-p}}e^{i\frac{\alpha_{1}}{\gamma_{1}}\frac{b|x|^{2}}{4(a+bt)}}u_{1},\ldots,(a+bt)^{\frac{2}{1-p}}e^{i\frac{\alpha_{l}}{\gamma_{l}}\frac{b|x|^{2}}{4(a+bt)}}u_{l}\right)\\
         &=(a+bt)^{\frac{2p}{1-p}}e^{i\frac{\alpha_{k}}{\gamma_{k}}\frac{b|x|^{2}}{4(a+bt)}}f_{k}\left(\mathbf{u}\right)\\
         &= \left[i\alpha_{k}\partial_{t}v_{k}^{A}+\gamma_{k}\Delta v_{k}^{A}\right](x,t),
     \end{split}
 \end{equation*}
 which complete the proof. 
\end{proof}

\begin{obs}\label{remsoluin}
Note that $\mathbf{u}$ is a solution of \eqref{system1} with $\beta_{k}=\frac{\alpha_k^2}{\gamma_k}$, $k=1,\ldots,l$, if and only if $\tilde{\mathbf{u}}$ given by
$$
\tilde{u}_k(x,t)=e^{i\frac{\alpha_k}{\gamma_k}t}u_k(x,t),  \qquad k=1,\ldots,l,
$$
is also  solution of \eqref{system1} but with $\beta_{k}=0$, $k=1,\ldots,l.$
\end{obs}

Now, let $\psib \in \mathcal{G}(1,\boldsymbol{0}) $. In particular $\psib$  is a solution of \eqref{system1} with $\beta_{k}=\frac{\alpha_k^2}{\gamma_k}$. Hence, from Remark \ref{remsoluin}, 
$$
\tilde{u}_k(x,t)=e^{i\frac{\alpha_k}{\gamma_k}t}\psib(x),  \qquad k=1,\ldots,l,
$$
is a solution of \eqref{system1} with $\beta_k=0$. Moreover, by Lemma \ref{pseudoconf}, for any $A\in SL(2,\R)$, $\mathbf{v}^A$ defined by
$$
 v_{k}^{A}(x,t)=(a+bt)^{\frac{2}{1-p}}e^{i\frac{\alpha_{k}}{\gamma_{k}}\frac{b|x|^{2}}{4(a+bt)}} e^{i\frac{\alpha_k}{\gamma_k}\frac{c+dt}{a+bt}}\psi_k\left(\frac{x}{a+bt}\right), \qquad k=1,\ldots,l,
$$
is also a solution. With this in hand we are able to establish the following.

 \subsection{Proof of Theorem \ref{thm:globalexistencecondn=51}}

\subsubsection{Existence of global solutions}
 
To extend globally, in time, the solutions we need an \textit{a priori} bound to its $H^1$-norm. Since the $L^2$-norm is conserved in time (see Theorem \ref{L2normconserved})  it is sufficient to get an \textit{a priori} bound for $K(\mathbf{u}(t))$.\\

We start by recalling the definition of the critical exponent, 
 \begin{equation}\label{critexp2}
    s_c= \frac{n}{2}-\frac{2}{p-1}.
 \end{equation}

\begin{proof}[Proof of Theorem \ref{thm:globalexistencecondn=51}.i)]
To obtain such a bound in the part (i), in Theorem \ref{thm:globalexistencecondn=51}  we start using    \eqref{conservationenergy} and the Gagliardo-Nirenberg-type inequality, \eqref{GNE2}, to write
\begin{equation*}
\begin{split}
K(\mathbf{u})&=E(\mathbf{u}_{0})-L(\mathbf{u})+2P(\mathbf{u})\leq E(\mathbf{u}_{0})+2\left|P(\mathbf{u})\right|\leq E(\mathbf{u}_{0})\\
&\quad+ 2C_{opt}Q(\mathbf{u}_{0})^{\frac{p-1}{2}(1-s_c)}K(\mathbf{u})^{\frac{p-1}{2}s_c+1},
\end{split}
    \end{equation*}
where we have used that $L(\mathbf{u})\geq 0$. Next, in the last term of the previous expression we can apply Young's inequality (noting that when $1<p<1+\frac{4}{n}$, then $\left(\frac{2}{(1-p)s_c},\frac{2}{(p-1)s_c+2}\right)$ is a pair of conjugate  exponents) to obtain for any $\epsilon>0$,

\begin{equation*}
    K(\mathbf{u}) \leq  E(\mathbf{u}_{0})+C_{\epsilon}Q(\mathbf{u}_{0})^{\frac{s_c -1}{s_c}}+\epsilon K(\mathbf{u}),
\end{equation*}
for some constant $C_\epsilon$. From this we deduce that  if we choose $0<\epsilon<1$, then 
\begin{equation}\label{GNE2.1}
K(\mathbf{u}(t))\leq (1-\epsilon)^{-1}\left[E(\mathbf{u}_{0})+C_{\epsilon}Q(\mathbf{u}_{0})^{\frac{s_c -1}{s_c}}\right],
\end{equation}
  which means that $K(\mathbf{u}(t))$ is uniformly bounded as we need.\\
 \end{proof}

\begin{proof}[Proof of Theorem \ref{thm:globalexistencecondn=51}.ii.a)]
Next, we prove the part ii.a) in  Theorem \ref{thm:globalexistencecondn=51}. In this case, when $p=1+\frac{4}{n}$ we have from \eqref{critexp2} that $s_c=0$. Thus,  from the energy expression \eqref{conservationenergy} and the Gagliardo-Nirenberg-type inequality \eqref{GNE2}, we have again
\begin{equation*}
K(\mathbf{u})\leq E(\mathbf{u}_{0})+2C_{opt}Q(\mathbf{u}_{0})^{\frac{2}{n}}K(\mathbf{u}),
\end{equation*}

Now, using the expression for the best constant appearing in Corollary  \ref{corollarybestconstant} with $p=1+\frac{4}{n}$ we have

\begin{equation*}
K(\mathbf{u})\leq E(\mathbf{u}_{0})+\left(\frac{Q(\mathbf{u}_{0})}{Q(\psib)}\right)^{\frac{2}{n}}K(\mathbf{u}),
\end{equation*}
which can be written as

\begin{equation}\label{positenerg1}
\left[1-\left(\frac{Q(\mathbf{u}_{0})}{Q(\psib)}\right)^{\frac{2}{n}}\right]K(\mathbf{u})\leq E(\mathbf{u}_{0}).
\end{equation}
Thus, if condition \eqref{L2GSCond} holds  then
$$K(\mathbf{u})\leq \left[1-\left(\frac{Q(\mathbf{u}_{0})}{Q(\psib)}\right)^{\frac{2}{n}}\right]^{-1}E(\mathbf{u}_{0}), $$
which implies the required bound for $K(\mathbf{u}(t))$.
\end{proof}

\subsubsection{Existence of blow up solutions}

\begin{proof}[Proof of Theorem \ref{thm:globalexistencecondn=51}.ii.b)]

Finally we prove part ii.b) in Theorem \ref{thm:globalexistencecondn=51}.

\begin{teore}\label{sharpn=4}
Let $\displaystyle p=1+\frac{4}{n}$ and $\psib \in \mathcal{G}(1,\boldsymbol{0}) $. For any $T>0$ let  $A=\left(\begin{array}{cc}
T&- 1 \\
0&\frac{1}{T}
\end{array}\right)$
in such a way that
\begin{equation*}
     v_{k}^{A}(x,t)=(T-t)^{\frac{2}{1-p}}e^{-i\frac{\alpha_{k}}{\gamma_{k}}\frac{|x|^{2}}{4(T-t)}+i\frac{\alpha_{k}}{\gamma_{k}}\frac{t}{T(T-t)}}\psi_{k}\left(\frac{x}{T-t}\right),\quad k=1,\ldots,l.
 \end{equation*}
Then
 \begin{enumerate}
     \item[(i)] $\mathbf{v}^A$  is a solution of\eqref{system1} with $\beta_{k}=0$ for $k=1,\ldots,l$.
     \item[(ii)] 
     \begin{equation*}
         v_{k}^{A}(x,0)=T^{\frac{2}{1-p}}e^{-i\frac{\alpha_{k}}{\gamma_{k}}\frac{|x|^{2}}{4T}}\psi_{k}\left(\frac{x}{T}\right),\quad k=1,\ldots,l.
     \end{equation*}
     \item[(iii)] $Q(\mathbf{v}^A(0))=Q(\psib)$.
     \item[(iv)] $K(\mathbf{v}^A(t))=O((T-t)^{-2})$ as $t\to T^{-}$.
 \end{enumerate}
\end{teore}

\begin{proof}
Statement (i) is a consequence of the Lemma \ref{pseudoconf}. Statements (ii), (iii) and (iv) follow from a direct calculation.
\end{proof}

\begin{coro}\label{corosharp}
Under the assumptions of Theorem \ref{sharpn=4}, if  $\mathbf{u}^A$ is defined by
$$
u_k^A(x,t)=e^{-i\frac{\alpha_k}{\gamma_k}t}v_k^A(x,t),
$$
then $\mathbf{u}^A$ is a solution of \eqref{system1} with $\beta_{k}=\frac{\alpha_k^2}{\gamma_k}$, $k=1,\ldots,l$,  such that $Q(\mathbf{u}^A(0))=Q(\psib)$  and $\mathbf{u}^A$ blows-up in finite time.
\end{coro}

\end{proof}

  \subsection{Proof of Theorem \ref{thm:globalexistencecondn=52}}

 \subsubsection{Existence of global solutions}

 \begin{proof}[Proof of Theorem \ref{thm:globalexistencecondn=52}.i)]

In order to prove the first part of Theorem \ref{thm:globalexistencecondn=52} we are going to apply Lemma \ref{supercritcalcase}. 

We start using again the conservation of energy, \eqref{conservationenergy} and the Gagliardo-Nirenberg-type inequality \eqref{GNE2}, to get
\begin{equation*}
K(\mathbf{u})\leq E(\mathbf{u}_{0})+2C_{opt}Q(\mathbf{u}_{0})^{\frac{p-1}{2}(1-s_c)}K(\mathbf{u})^{\frac{p-1}{2}s_c+1},
\end{equation*}

Now,  we define 

\begin{equation}\label{notation1}
    a=E(\mathbf{u}_{0}),\quad b=2C_{opt}Q(\mathbf{u}_{0})^{\frac{p-1}{2}(1-s_c)},\quad q=\frac{p-1}{2}s_c+1\quad \mbox{and } \quad G(t)=K(\mathbf{u}(t))
\end{equation}
With this notation we have
\begin{equation}\label{notation2}
    \gamma:=(bq)^{-\frac{1}{q-1}}=\frac{n}{2(1-s_c)}\frac{Q(\psib)^ {\frac{1}{s_c}}}{Q(\mathbf{u}_{0})^{\frac{1-s_c}{s_c}}}.
\end{equation}

Next, using  Lemma \ref{identitiesfunctionals} we obtain $K(\psib)=\frac{n}{2(1-s_c)}Q(\psib)$ and $\mathcal{E}(\psib)=\frac{s_c}{1-s_c}Q(\psib)$. Putting these facts together,  we can show that $a<\left(1-\frac{1}{q}\right)\gamma$ is equivalent to \eqref{conditionsharp1} and $G(0)<\gamma$ is equivalent to \eqref{conditionsharp2}. Thus an application of  Lemma \ref{supercritcalcase} gives the desired bound to  $K(\mathbf{u}(t))$ and as before, this means that the solution can be extended globally in time, which finished the proof of Theorem \ref{thm:globalexistencecondn=52}.i). 
 \end{proof}

\subsubsection{Existence of blow up solutions}

Now we will prove the second part of Theorem \ref{thm:globalexistencecondn=52}. Since by assumption   system \eqref{system1} satisfies the mass resonance condition, we can apply the Virial identities stated in Theorem \ref{persistL2wheit} and Corollary \ref{corovirrad}  to construct blow-up solutions in the the $L^{2}$ supercritical and $H^{1}$ subcritical case; $1+\frac{4}{n}<p<\frac{n+2}{n-2}$.  To do so, we consider suitable initial data which does not satisfy condition \eqref{conditionsharp2} and whose corresponding solution blows-up in finite time.  We follow the method presented in   \cite{Holmer2} and \cite{Ogawa}  which was applied to a single Schr\"{o}dinger equation. In the case of  systems in \cite{Pastor}  an adaptation appears.  

\begin{proof}[Proof of Theorem \ref{thm:globalexistencecondn=52}.ii)] First of all, we apply Lemma \ref{corosupercritcalcase} to obtain a modification  of condition \eqref{gradientcondblowup}. In fact, 
since  condition \eqref{conditionsharp1} holds,  we can find a sufficient small constant  $\delta_{1}>0$  such that
\begin{equation}\label{conddelta1}
Q(\mathbf{u}_0)^{1-s_c}E(\mathbf{u}_0)^{s_c}<(1-\delta_{1})^{s_c}Q(\psib)^{1-s_c}\mathcal{E}(\psib)^{s_c}.
\end{equation}
With the same notation as in \eqref{notation1} and \eqref{notation2}, we can check that $G(0)>\gamma$ is equivalent to \eqref{gradientcondblowup} and $a<(1-\delta_{1})\left(1-\frac{1}{q}\right)\gamma$ is equivalent to \eqref{conddelta1}. Thus Lemma 
 \ref{corosupercritcalcase} implies that there exists $\delta_{2}>0$ such that
\begin{equation}\label{conddelta2}
Q(\mathbf{u}_0)^{1-s_c}K(\mathbf{u}(t))^{s_c}>(1+\delta_{2})^{s_c}Q(\psib)^{1-s_c}K(\psib)^{s_c}.
\end{equation}
Next, from \eqref{critexp2} and since we are assuming $p>1+\frac{4}{n}$, we have
\begin{equation}\label{sccond}
    \frac{2(p-1)s_c+4}{n}=p-1>0 \quad\mbox{and} \quad s_c>0. 
\end{equation}

We first assume  $x\mathbf{u}_0\in \mathbf{L}^{2}$. Hence, the variance, \eqref{fuctV}, is finite and we recall  that the second derivative of the variance, \eqref{secderV}, is given by
\begin{equation*}
V''(t)=2n(p-1) E(\mathbf{u}_0)-2n(p-1)L(\mathbf{u}(t))-2[n(p-1)-4]K(\mathbf{u}(t)),\qquad t\in I.
\end{equation*}
By using \eqref{critexp2} and noting that  $L(\mathbf{u}(t))\geq 0$ and  we have

\begin{equation}\label{virialidentity2}
V''(t)\leq 2[2(p-1)s_c+4] E(\mathbf{u}_0)-4(p-1)s_c K(\mathbf{u}(t)),\qquad t\in I.
\end{equation}
Multiplying both sides of (\ref{virialidentity2})  by $Q(\mathbf{u}_0)^{\frac{1-s_c}{s_c}}$, using the inequalities \eqref{conddelta1}-\eqref{conddelta2}, the relation $\mathcal{E}(\psib)=\frac{2s_c}{n}K(\psib)$ and \eqref{sccond} we obtain, for any $t\in I$,
\begin{equation*}
\begin{split}
V''(t)Q(\mathbf{u}_0)^{\frac{1-s_c}{s_c}}&=2[2(p-1)s_c+4] E(\mathbf{u}_0)Q(\mathbf{u}_0)^{\frac{1-s_c}{s_c}}\\
&\quad-4(p-1)s_c K(\mathbf{u}(t))Q(\mathbf{u}_0)^{\frac{1-s_c}{s_c}}\\
&<2[2(p-1)s_c+4](1-\delta_{1})\mathcal{E}(\psib)Q(\psib)^{\frac{1-s_c}{s_c}} \\
&\quad-4(p-1)s_c (1+\delta_{2}) K(\psib)Q(\psib)^{\frac{1-s_c}{s_c}}\\
&= 4(1-\delta_{1}) \frac{2(p-1)s_c+4}{n}s_c K(\psib)Q(\psib)^{\frac{1-s_c}{s_c}}\\
&\quad-4(p-1)s_c (1+\delta_{2})K(\psib)Q(\psib)^{\frac{1-s_c}{s_c}}\\
&=-4(p-1)s_c\left(\delta_{1}+\delta_{2}\right)K(\psib)Q(\psib)^{\frac{1-s_c}{s_c}}\\
&=:-B,
\end{split}
\end{equation*}
where $B$ is a positive constant, by \eqref{sccond}.
Thus, if we assume that  $I$ is infinite must exist $t^{*}\in I$ such that $V(t^{*})<0$, which is a contradiction,  because $V>0$. Therefore $I$ must be finite.\\

Next, we assume that $\mathbf{u}_0$ is radially symmetric. In this case, we can use the localized version of the virial identity given in \eqref{secondervradialcase}. Thus,   we use $\chi_{R}$  defined in Lemma \ref{lemafunctionchi} 
 as $\varphi$ in \eqref{secondervradialcase}, to get
 \begin{equation}\label{secodnderivativeVwithchi}
 \begin{split}
V''(t)&=2\int \chi_{R}''\left(\sum_{k=1}^{l}\gamma_{k}|\nabla u_{k}|^{2}\right)\;dx-\frac{1}{2}\int\Delta^{2}\chi_{R}\left(\sum_{k=1}^{l}\gamma_{k}|u_{k}|^{2}\right)\;dx\\&\quad+(1-p)\mathrm{Re}\int\Delta\chi_{R} F\left(\mathbf{u}\right)\;dx.
\end{split}
\end{equation}
The idea is again to apply a convex argument, so we will estimate each one of the terms appearing in $V''$ and conclude that it is less than a negative constant. For the first term, we use the fact that $\chi_{R}''(r)\leq 2$,  to obtain
\begin{equation}\label{estim1}
2\int \chi_{R}''\left(\sum_{k=1}^{l}\gamma_{k}|\nabla u_{k}|^{2}\right)\;dx\leq 4\sum_{k=1}^{l}\gamma_{k}\|\nabla u_{k}\|^{2}_{L^{2}}= 4K(\mathbf{u}).
\end{equation}
For the second one, using Lemma \ref{lemafunctionchi} and the fact that the  charge is a conserved quantity we get
\begin{equation}\label{estim2}
\begin{split}
-\frac{1}{2}\int\Delta^{2}\chi_{R}\left(\sum_{k=1}^{l}\gamma_{k}|u_{k}|^{2}\right)dx&=-
\frac{1}{2}\int_{\{|x|\geq R\}}\Delta^{2}\chi_{R}\left(\sum_{k=1}^{l}\gamma_{k}|u_{k}|^{2}\right)dx\\
&\leq\frac{C_{2}}{R^{2}}\int_{\{|x|\geq R\}}\left(\sum_{k=1}^{l}\gamma_{k}|u_{k}|^{2}\right)dx\\
&\leq\frac{C_{2}}{R^{2}}\max_{1\leq j\leq l}\left\{\frac{\gamma_{j}^{2}}{\alpha_{j}^{2}}\right\}\sum_{k=1}^{l}\frac{\alpha_{k}^{2}}{\gamma_{k}}\|u_{k}\|^{2}_{L^{2}}\\
&=\frac{{C}_{2}'}{R^{2}}Q(\mathbf{u}_0),
\end{split}
\end{equation}
for some positive constant $C_2'$.
Finally, the last term in  (\ref{secodnderivativeVwithchi}) is estimated as follows
\begin{equation*}
\begin{split}
(1-p)\mbox{Re}\int\Delta\chi_{R} F\left(\mathbf{u}\right)\;dx&=(1-p)\mathrm{Re}\int_{\{|x|\leq R\}}\Delta\chi_{R} F\left(\mathbf{u}\right)\;dx\\
&\quad+(1-p)\mathrm{Re}\int_{\{|x|\geq R\}}\Delta\chi_{R} F\left(\mathbf{u}\right)\;dx\\
&\leq 2n(1-p)\;\mathrm{Re}\int_{\{|x|\leq R\}} F\left(\mathbf{u}\right)\;dx\\
&\quad+C_{1}\int_{\{|x|\geq R\}}\left|\mathrm{Re}\,F\left(\mathbf{u}\right)\right|\;dx\\
&=2n(1-p)\;\mathrm{Re}\,\int_{\R^{n}} F\left(\mathbf{u}\right)\;dx+C_{1}'\int_{\{|x|\geq R\}}\left|\mathrm{Re}\,F\left(\mathbf{u}\right)\right|\;dx\\
&= 2n(1-p) P(\mathbf{u})+C_{1}'\int_{\{|x|\geq R\}}\left|\mathrm{Re}\,F\left(\mathbf{u}\right)\right|\;dx,
\end{split}
\end{equation*}
where we have used again   Lemma  \ref{lemafunctionchi}. Here $C_1'$ is also a positive constant.
Now, from definition of the energy \eqref{conservationenergy} and the fact that it is conserved we can write for the first
term  in the last inequality
$2n(1-p) P(\mathbf{u})=n(p-1)E(\mathbf{u}_0)+n(1-p)K(\mathbf{u})+n(1-p)L(\mathbf{u}).$
Thus, by using \eqref{critexp2} and  that  $L(\mathbf{u}(t))\geq 0$ we have

\begin{equation*}
    2n(1-p) P(\mathbf{u}) \leq [2(p-1)s_c+4]E(\mathbf{u}_0)-[2(p-1)s_c+4]K(\mathbf{u}).
\end{equation*}
Which implies that

\begin{equation}\label{estim3}
\begin{split}
(1-p)\mathrm{Re}\int\Delta\chi_{R} F\left(\mathbf{u}\right)\;dx&\leq  [2(p-1)s_c+4]E(\mathbf{u}_0)- [2(p-1)s_c+4]K(\mathbf{u})\\
&\quad+C_{1}'\int_{\{|x|\geq R\}}\left|\mathrm{Re}\, F\left(\mathbf{u}\right)\right|\;dx.
\end{split}
\end{equation}
Gathering together \eqref{secodnderivativeVwithchi}-(\ref{estim3}) and using Lemma \ref{estdifF}, we have
\begin{equation}\label{1b}
\begin{split}
V''(t)&\leq [2(p-1)s_c+4]E(\mathbf{u}_0)-2(p-1)s_cK(\mathbf{u})+\frac{{C'}_{2}}{R^{2}}Q(\mathbf{u}_0)\\
&\quad+C_{1}'\int_{\{|x|\geq R\}}\left|\mbox{Re}\,F\left(\mathbf{u}\right)\right|\;dx\\
&\leq [2(p-1)s_c+4]E(\mathbf{u}_0)-2(p-1)s_cK(\mathbf{u})+\frac{{C'}_{2}}{R^{2}}Q(\mathbf{u}_0)\\
&\quad+C_{1}'C\sum_{k=1}^{l}\|u_{k}\|^{p+1}_{L^{p+1}(|x|\geq R)}.
\end{split}
\end{equation}
Finally, to deal with the last term in \eqref{1b} we use Lemma \ref{StraussLemaconsequence} and taking into account the assumption of $p$: $1+\frac{4}{n}<p<\min\left\{\frac{n+2}{n-2},5\right\}$ we apply  Young's inequality with $\epsilon$ (small) to conclude that
 \begin{equation*}
 \begin{split}
\sum_{k=1}^{l}\|u_{k}\|^{p+1}_{L^{p+1}(|x|\geq R)}
&\leq \frac{\tilde{C}}{R^{\frac{(n-1)(p-1)}{2}}}\sum_{k=1}^{l}\|u_{k}\|^{(p+3)/2}_{L^{2}(|x|\geq R)}\|\nabla u_{k}\|^{(p-1)/2}_{L^{2}(|x|\geq R)}\\
&\leq C_{\epsilon}\sum_{k=1}^{l}\left\{R^{-{\frac{(n-1)(p-1)}{2}}}\left[\left(\frac{\alpha_{k}^{2}}{\gamma_{k}}\right)^{1/2}\|u_{k}\|_{L^{2}(|x|\geq R)}\right]^{(p+3)/2}\right\}^{4/(5-p)} \\
&\quad+\epsilon\sum_{k=1}^{l}\left\{\left[\gamma_{k}^{1/2}\|\nabla u_{k}\|_{L^{2}(|x|\geq R)}\right]^{(p-1)/2}\right\}^{4/(p-1)}\\
&\leq \frac{\tilde{C}_{\epsilon}}{R^{\frac{2(n-1)(p-1)}{5-p}}}\left(\sum_{k=1}^{l}\frac{\alpha_{k}^{2}}{\gamma_{k}}\|u_{k}\|^{2}_{L^{2}(|x|\geq R)}\right)^{(p+3)/(5-p)}\\
&\quad+\epsilon\sum_{k=1}^{l}\gamma_{k}\|\nabla u_{k}\|^{2}_{L^{2}(|x|\geq R)}\\
&\leq\frac{\tilde{C}_{\epsilon}}{R^{\frac{2(n-1)(p-1)}{5-p}}}Q(\mathbf{u}_0)^{(p+3)/(5-p)}+\epsilon K(\mathbf{u}).
\end{split}
 \end{equation*}
 Therefore, from \eqref{1b},
\begin{equation}
\begin{split}\label{estimaV2da}
V''(t)&\leq [2(p-1)s_c+4]E(\mathbf{u}_0)+\left[-2(p-1)s_c+\epsilon\right]K(\mathbf{u})+\frac{C_2'}{R^{2}}Q(\mathbf{u}_0)
\\
&\quad+\frac{\tilde{C}_{\epsilon}}{R^{\frac{2(n-1)(p-1)}{5-p}}}Q(\mathbf{u}_0)^{(p+3)/(5-p)}.
\end{split}
\end{equation}
Multiplying (\ref{estimaV2da}) by $Q(\mathbf{u}_0)^{\frac{1-s_c}{s_c}}$, we obtain
\begin{equation*}
\begin{split}
Q(\mathbf{u}_0)^{\frac{1-s_c}{s_c}}V''(t)&\leq [2(p-1)s_c+4]E(\mathbf{u}_0)Q(\mathbf{u}_0)^{\frac{1-s_c}{s_c}}\\
&\quad+\left[-2(p-1)s_c+\epsilon\right]K(\mathbf{u})Q(\mathbf{u}_0)^{\frac{1-s_c}{s_c}}
+\frac{C_2'}{R^{2}}Q(\mathbf{u}_0)^{\frac{1}{s_c}}\\
&\quad+\frac{\tilde{C}_{\epsilon}}{R^{\frac{2(n-1)(p-1)}{5-p}}}Q(\mathbf{u}_0)^{\frac{p+3}{5-p}+\frac{1-s_c}{s_c}}.
\end{split}
\end{equation*}
 Using (\ref{conddelta1}), (\ref{conddelta2}), the relation $\mathcal{E}(\psib)=\frac{2s_c}{n}K(\psib)$ and \eqref{sccond} we can write
\begin{equation*}
\begin{split}
Q(\mathbf{u}_0)^{\frac{1-s_c}{s_c}}V''(t)&\leq [2(p-1)s_c+4](1-\delta_{1}) Q(\psib)^{\frac{1-s_c}{s_c}}\mathcal{E}(\psib)\\
&\quad+\left[-2(p-1)s_c+\epsilon\right](1+\delta_{2})Q(\psib))^{\frac{1-s_c}{s_c}}K(\psib) +\frac{C_2'}{R^{2}}Q(\mathbf{u}_0)^{\frac{1}{s_c}}
\\
&\quad+\frac{\tilde{C}_{\epsilon}}{R^{\frac{2(n-1)(p-1)}{5-p}}}Q(\mathbf{u}_0)^{\frac{p+3}{5-p}+\frac{1-s_c}{s_c}}\\
&=\left[-2s_{c}(p-1)(\delta_{1}+\delta_{2})+\epsilon(1+\delta_{2})\right]Q(\psib)^{\frac{1-s_c}{s_c}}K(\psib)\\
&\quad+
\frac{C_2'}{R^{2}}Q(\mathbf{u}_0)^{\frac{1}{s_c}}+\frac{\tilde{C}_{\epsilon}}{R^{\frac{2(n-1)(p-1)}{5-p}}}Q(\mathbf{u}_0)^{\frac{p+3}{5-p}+\frac{1-s_c}{s_c}}.
\end{split}
\end{equation*}
 Choosing $\epsilon>0$ small enough and  $R>0$ sufficiently large , we can conclude that $V''(t)<-B$, for some constant $B>0$, since $2s_{c}(p-1)>0$, by  \eqref{sccond}. As in the case of finite variance, we then conclude that $I$ must be finite. 
\end{proof}

\section{Particular cases}\label{sec.parti.cases}

In this section we verify some of the results obtained in the previous sections applying to some particular nonlinear Schrodinger systems.

\subsection{Quadratic system}
We consider again system \eqref{system1J}, 
\begin{equation}\label{system1J2}
\begin{cases}
\displaystyle i\partial_{t}u+\Delta u=-2\overline{u}v,\\
\displaystyle i\partial_{t}v+\kappa\Delta v=- u^{2},
\end{cases}
\end{equation}
where $\ka>0$. System \eqref{system1J2} is  a particular model of system \eqref{system1} when $l=2$ and $p=2$. In this, case we have  the following parameters

\begin{center}
    \begin{tabular}{|c|c|}\hline 
     $\alpha_1 =1$  &  $\gamma_1 =1$   \\\hline 
       $\alpha_2 =1$  &  $\gamma_2 =\ka$  \\\hline 
    \end{tabular}
\end{center}

\subsubsection{The nonlinearities}

From  \eqref{critexp} we have that the critical exponent is $s_c=\frac{n}{2}-2$, thus system \eqref{systemFA2} satisfies the following regimes 

\begin{equation*}
     L^{2}-
     \begin{cases}
     \mbox{subcritical}, \quad\mbox{if}\quad  1\leq n\leq 3,\\
     \mbox{critical},\quad\mbox{if}\quad n=4,\\
     \mbox{supercritical},\quad\mbox{if}\quad n\geq 5;
     \end{cases}\quad \mbox{and} \quad
     H^{1}-
     \begin{cases}
     \mbox{subcritical},\quad \mbox{if }\quad 1\leq n\leq 5   \\
     \mbox{critical},\quad\mbox{if}\quad n=6,\\
     \mbox{supercritical},\quad\mbox{if}\quad n\geq 7.
     \end{cases}
 \end{equation*}

The nonlinear functions in system \eqref{system1J2} satisfying   assumptions \ref{H1} and \ref{H2*} are given by

\begin{equation}\label{fksyst2}
    f_1(z_1,z_2)=2\overline{z}_1 z_2 \qquad \mbox{and}\qquad f_2(z_1,z_2)= z_{1}^{2}.
\end{equation}

We now turn to the expression of the potential function, $F$, appearing in assumption \ref{H3}, which for system \eqref{system1J2} has the form 

\begin{equation}\label{Fsyst2}
    F(z_1,z_2)=\overline{z}_1^2z_2. 
\end{equation}

We note that $F$  is a  homogeneous function of degree $3$, which means that assumption \ref{H5*} is satisfied and it is  also clear that $F$ satisfies \ref{H6}. We also observe that if we take $\sigma_1=2$ and $\sigma_2=4$,  functions $f_1$ and $f_2$ in \eqref{fksyst2} satisfy \ref{H4*}. \\

Finally, we note that the function $F$ appearing in \eqref{Fsyst2} satisfies assumptions \ref{H7} and \ref{H8}. For\ref{H8}, note that the potential $F$ can be decomposed as a sum of  functions $F_s$ satisfying \eqref{condsupmod}

\subsubsection{Ground state solutions}

Now, to show the existence of \textit{ground states} solutions for $ 1\leq n\leq 5$, in Lemma \ref{identitiesfunctionals} we presented some relations between the functionals $I$, $Q$
, $K$ and $P$. In the case of system \eqref{system1J2} we have (see \cite[Theorem 4.1]{Hayashi}, also Lemmas 2.3 and 2.5 in \cite{NoPa})

\begin{equation*}
    P(\psib)=2I, \qquad K(\psib)=nI, \qquad \mathcal{Q}(\psib)=(6-n)I(\psib) \qquad \mbox{and}
\end{equation*}
\begin{equation*}
    J(\psib)=\frac{n^{\frac{n}{4}}}{2}\left(6-n\right)^{\frac{6-n}{4}}I(\boldsymbol{\psi})^{\frac{1}{2}}.
\end{equation*}

As we saw in Corollary \ref{corollarybestconstant}, one consequence of the existence of \textit{ground states} solutions was to determine the best constant in a    Gagliardo-Nirenberg-type inequality. In the case of system \eqref{system1J2} this constant is given by

\begin{equation}\label{bestconstant2}
C_{opt}=\frac{2\left(6-n\right)^{\frac{n-4}{4}}}{n^{\frac{n}{4}}}\frac{1}{\mathcal{Q}(\tilde{\psib})^{\frac{1}{2}}},
\end{equation}
where $\tilde{\psib}\in \mathcal{G}(\omega,\boldsymbol{\beta})$.

\subsubsection{Mass-resonance}
Regarded to the mass-resonance condition (see \eqref{RC}) we have that for system \eqref{system1J2}, $\displaystyle\frac{\alpha_1}{\gamma_1}=1$ and $\displaystyle\frac{\alpha_2}{\gamma_2}=\sigma$. Then, the condition   \eqref{RC} is equivalent to

 \begin{equation*}
      \left(1-\frac{1}{2\kappa}\right)\mathrm{Im}(u^{2}\overline{v})=0.
  \end{equation*}
Thus, system  \eqref{system1J2} satisfies  the mass-resonance condition if and only if $\kappa=\frac{1}{2}$, which coincides with the value considered in the literature (see discussion before    \eqref{RC}). 

\subsubsection{About Theorems \ref{thm:globalexistencecondn=51} and \ref{thm:globalexistencecondn=52}}
For system \eqref{system1J2}, the global existence result presented in Theorem \ref{thm:globalexistencecondn=51} part (i) was proved in reference \cite[Theorem 3.7]{Hayashi}. Parts ii.a) and ii.b) were shown in \cite[Theorem 6.1]{Hayashi} and \cite[Theorem 6.3]{Hayashi}, respectively. \\
In reference \cite{NoPa2}, the same result was proved for  a more general quadratic system (see \cite[Theorem 3.16.(i)]{NoPa2}, \cite[Theorem 5.2.(i)]{NoPa2}, and \cite[Theorem 5.14]{NoPa2}). 

The corresponding proof of Theorem \ref{thm:globalexistencecondn=52} for system \eqref{system1J2} was proved in reference \cite[Theorems 1.1 and 1.2]{NoPa}, see also reference \cite{hamano2018global}. Again, for a more general nonlinearities with quadratic grwouth, this result can be founded in reference \cite[Theorem 5.2.(ii)]{NoPa2}, for part (i) and \cite[Theorem 5.16]{NoPa2} for part (ii).

\subsection{Cubic system} 

We recall system \eqref{systemFA}, 
\begin{equation}\label{systemFA2}
\begin{cases}
\displaystyle i\partial_{t}u+\Delta u-u+\left(\frac{1}{9}|u|^{2}+2|w|^{2}\right)u+\frac{1}{3}\overline{u}^{2}w=0,\\
\displaystyle i\sigma\partial_{t}w+\Delta w-\mu w+ (9|w|^{2}+2|u|^{2})w+\frac{1}{9}u^{3}=0,
\end{cases}
\end{equation}
where $\sigma,\mu>0$.

This is a cubic nonlinear Schr\"{o}dinder system, which  corresponds to a particular case of system \eqref{system1} when $l=2$ and $p=3$. The parameters of the system are given by

\begin{center}
    \begin{tabular}{|c|c|c|}\hline 
     $\alpha_1 =1$  &  $\gamma_1 =1$ &  $\beta_1 =1$  \\\hline 
       $\alpha_2 =\sigma$  &  $\gamma_2 =1$ &  $\beta_2 =\mu$ \\\hline 
    \end{tabular}
\end{center}

\subsubsection{The nonlinearities}

From  \eqref{critexp} the critical exponent is $s_c=\frac{n}{2}-1$, thus system \eqref{systemFA2} satisfies the following regimes 

\begin{equation*}
     L^{2}-
     \begin{cases}
     \mbox{subcritical}, \quad\mbox{if}\quad n=1,\\
     \mbox{critical},\quad\mbox{if}\quad n=2,\\
     \mbox{supercritical},\quad\mbox{if}\quad n\geq 3;
     \end{cases}\quad \mbox{and} \quad
     H^{1}-
     \begin{cases}
     \mbox{subcritical},\quad \mbox{if }\quad 1\leq n\leq 3   \\
     \mbox{critical},\quad\mbox{if}\quad n=4,\\
     \mbox{supercritical},\quad\mbox{if}\quad n\geq 5.
     \end{cases}
 \end{equation*}

The nonlinear functions satisfying  assumptions \ref{H1} and \ref{H2*} are

\begin{equation}\label{fksyst3}
    f_1(z_1,z_2)=\left(\frac{1}{9}|z_1|^{2}+2|z_2|^{2}\right)z_1+\frac{1}{3}\overline{z}_1^{2}z_2 \qquad \mbox{and}\qquad f_2(z_1,z_2)= (9|z_2|^{2}+2|z_1|^{2})z_2+\frac{1}{9}z_{1}^{3}.
\end{equation}

On the other hand, the potential function $F$ appearing in assumption \ref{H3} is given by

\begin{equation}\label{Fsyst3}
    F(z_1,z_2)=\frac{1}{36}|z_1|^4+\frac{9}{4}|z_2|^4+|z_1|^2|z_2|^2+\frac{1}{9}\overline{z}_1^3z_2. 
\end{equation}

This functions is homogeneous of degree $4$, which means that assumption \ref{H5} is satisfied and it is  also clear that $F$ satisfies \ref{H6}.

Observe that taking $\sigma_1=2$ and $\sigma_2=6$, the nonlinear functions, \eqref{fksyst3}, satisfy \ref{H4*}. \\

Finally, we note that the function $F$ appearing in \eqref{Fsyst3} satisfies assumptions \ref{H7} and \ref{H8}. For \ref{H8}, note that the potential $F$ can be decomposed as a sum of  functions $F_s$ satisfying \eqref{condsupmod}

\subsubsection{Ground state solutions}

Existence of \textit{ground states} solutions need  Lemma \ref{identitiesfunctionals}, which  presented some relations between the functionals $I$, $Q$
, $K$ and $P$. In the case of system \eqref{systemFA2} these functionals  become

\begin{equation*}
    P(\psib)=I, \qquad K(\psib)=nI \qquad \mbox{and}\qquad J(\psib)=n^{\frac{n}{2}}(4-n)^{2-\frac{n}{2}}I(\psib).
\end{equation*}
They are consistent with those appearing in Lemma 3.4 in reference \cite{oliveira2018schr}.\\

As a consequence of the existence of ground states solutions we have the best constant in a Gagliardo-Nirenberg-type inequality, which in the case of system \eqref{systemFA2} has the form

\begin{equation*}
C_{opt}=\frac{(4-n)^{\frac{n}{2}-1}}{n^\frac{n}{2}}\frac{1}{\mathcal{Q}(\psib)},
\end{equation*}
where $\psib$ is any \textit{ground state} solution. This is precisely the assertion of   Corollary 3.6 in \cite{oliveira2018schr}. 

\subsubsection{Mass-resonance}

Concerning  mass-resonance condition, \eqref{RC}, we have that for system \eqref{systemFA2}, $\displaystyle\frac{\alpha_1}{\gamma_1}=1$ and $\displaystyle\frac{\alpha_2}{\gamma_2}=\sigma$. Then, condition   \eqref{RC} is equivalent to

\begin{equation*}
    \left(\frac{1}{3}-\frac{\sigma}{9}\right)\mbox{Im}(\overline{z}_{1}^3z_2)=0.
\end{equation*}
Hence, system  \eqref{systemFA2} satisfies the mass-resonance condition if and only if $\sigma=3$. This value coincides with the one considered, for example in reference \cite{oliveira2018schr}.

\subsubsection{About Theorems \ref{thm:globalexistencecondn=51} and \ref{thm:globalexistencecondn=52}}

For  system \eqref{systemFA2}, with cubic nonlinearities, the proof of Theorem \ref{thm:globalexistencecondn=51} part (i) can be founded in \cite[Corollary 3.2]{oliveira2018schr}, part ii.a) is proved in \cite[Corollary 3.3]{oliveira2018schr} and part ii.b) in the same reference, in Theorem 4.2. The authors also proved Theorem \ref{thm:globalexistencecondn=52} part (i) and (ii) in Theorem 3.10 and Theorem 4.6, respectively.

\section*{Acknowledgement}

N.N. was partially supported by Universidad de Costa Rica, through project number C1147.

\end{document}